\documentclass[11pt]{amsart}
\usepackage{amsxtra}
\usepackage{amssymb}
\usepackage[mathscr]{eucal}
\usepackage{amssymb,amscd,amsthm,amsmath,graphicx,color,mathtools} 
\usepackage{booktabs}
\usepackage{color}
\usepackage[color,matrix,arrow]{xy}

\usepackage{todonotes}
\theoremstyle{plain}

\newtheorem{theorem}{Theorem}[section]
\newtheorem{lemma}[theorem]{Lemma}
\newtheorem{definition-theorem}[theorem]{Definition-Theorem}
\newtheorem{proposition}[theorem]{Proposition}
\newtheorem{corollary}[theorem]{Corollary}
\newtheorem{definition}[theorem]{Definition}
\newtheorem{example}[theorem]{Example}
\newtheorem{remark}[theorem]{Remark}
\newtheorem{conjecture}[theorem]{Conjecture}
\newtheorem{notation}[theorem]{Notation}

\newtheorem*{maintheorem*}{Main Theorem}
\newcommand \bth[1] { \begin{theorem}\label{t#1} }
\newcommand \ble[1] { \begin{lemma}\label{l#1} }

\newcommand \bpr[1] { \begin{proposition}\label{p#1} }
\newcommand \bco[1] { \begin{corollary}\label{c#1} }
\newcommand \bde[1] { \begin{definition}\label{d#1}\rm }
\newcommand \bex[1] { \begin{example}\label{e#1}\rm }
\newcommand \bre[1] { \begin{remark}\label{r#1}\rm }
\newcommand \bcj[1] { \begin{conjecture}\label{j#1}\rm }

\newcommand \bnota[1] { \begin{notation}\label{n#1}\rm }
\renewcommand {\eth} { \end{theorem} }
\newcommand {\ele} { \end{lemma} }

\newcommand {\epr} { \end{proposition} }
\newcommand {\eco} { \end{corollary} }
\newcommand {\ede} { \end{definition} }
\newcommand {\eex} { \end{example} }
\newcommand {\ere} { \end{remark} }
\newcommand {\ecj} { \end{conjecture} }

\newcommand {\enota} { \end{notation} }

\newcommand \leref[1]{Lemma \ref{l#1}}
\newcommand \prref[1]{Proposition \ref{p#1}}
\newcommand \coref[1]{Corollary \ref{c#1}}

\newcommand \deref[1]{Definition \ref{d#1}}
\newcommand \exref[1]{Example \ref{e#1}}

\newcommand \notaref[1]{Notation \ref{n#1}}

\makeatletter
\newsavebox{\@brx}
\newcommand{\llangle}[1][]{\savebox{\@brx}{\(\m@th{#1\langle}\)}%
  \mathopen{\copy\@brx\kern-0.5\wd\@brx\usebox{\@brx}}}
\newcommand{\rrangle}[1][]{\savebox{\@brx}{\(\m@th{#1\rangle}\)}%
  \mathclose{\copy\@brx\kern-0.5\wd\@brx\usebox{\@brx}}}
\makeatother

\usepackage{tikz}
\usetikzlibrary{matrix}
\usepackage{graphicx,ctable,booktabs}
\usetikzlibrary{arrows.meta}
\usepackage{tikz-cd}

\usepackage{hyperref}
\usepackage{cleveref}

\DeclareMathOperator{\Ext}{Ext}

\DeclareMathOperator{\Spec}{Spec}

 \DeclareMathOperator{\Proj}{Proj}

 \DeclareMathOperator{\Hom}{Hom}

\DeclareMathOperator{\Mod}{{\sf Mod}}

\DeclareMathOperator{\der}{{\sf D}}
\DeclareMathOperator{\dperf}{{\sf D^{perf}}}

\DeclareMathOperator{\supp}{supp}

\DeclareMathOperator{\Spc}{Spc}

\DeclareMathOperator{\Ideals}{Ideals}

\DeclareMathOperator{\cone}{cone}
\DeclareMathOperator{\Ind}{Ind}

\DeclareMathOperator{\PHom}{PHom}

\DeclareMathOperator{\Supp}{Supp}
\DeclareMathOperator{\Thom}{Thom}

\newcommand{\mf}{\mathfrak}
\newcommand{\mc}{\mathcal}

\newcommand{\id}{\operatorname{id}}

\newcommand{\kk}{\Bbbk}

\newcommand{\bR}{\mathbf R}
\newcommand{\bS}{\mathbf S}
\newcommand{\bC}{\mathbf C}
\newcommand{\bK}{\mathbf K}
\newcommand{\bL}{\mathbf L}
\newcommand{\bM}{\mathbf M}
\newcommand{\bP}{\mathbf P}
\newcommand{\bQ}{\mathbf Q}

\newcommand{\bI}{\mathbf I}
\newcommand{\bJ}{\mathbf J}

\newcommand{\Loc}{\operatorname{Loc}}

\newcommand{\unit}{\ensuremath{\mathbf 1}}

\newcommand{\ul}{\underline}

\newcommand{\suppB}{\supp^{\rm B}}
\newcommand{\suppC}{\supp^{\rm C}}
\newcommand{\SuppB}{\Supp^{\rm B}}
\newcommand{\SuppC}{\Supp^{\rm C}}

\newcommand{\itGamm}{{\it{\Gamma}}}

\newcommand{\cohom}{\operatorname{H}^\bullet}

\newcounter{listequation}

\numberwithin{equation}{section}

\title[Formal extension of noncommutative tensor-triangular support varieties]{Formal extension of noncommutative tensor-triangular support varieties}

\begin{document}

\author[Cai]{Merrick Cai}
\address{(Cai) Department of Mathematics, Harvard University, Cambridge, MA 02139, U.S.A.}
\email{merrickcai@math.harvard.edu}

\author[Vashaw]{Kent B. Vashaw}
\address{(Vashaw) Department of Mathematics,
University of California Los Angeles,
Los Angeles, CA 90095, U.S.A.}
\email{kentvashaw@math.ucla.edu}



\begin{abstract}
Given a support variety theory defined on the compact part of a monoidal triangulated category, we define an extension to the non-compact part following the blueprint of Benson--Carlson--Rickard, Benson--Iyengar--Krause, Balmer--Favi, and Stevenson. We generalize important aspects of the theory of extended support varieties to the noncommutative case, and give characterizations of when an extended support theory detects the zero object, under certain assumptions. In particular, we show that when the original support variety theory is based on a Noetherian topological space, detects the zero object,  satisfies a generalized tensor product property, and comes equipped with a comparison map, then the extended support variety also detects the zero object. In the case of stable categories of finite tensor categories, this gives conditions under which the central cohomological support admits an extension that detects the zero object, confirming part of a recent conjecture made by the second author together with Nakano and Yakimov.
\end{abstract}

\maketitle



\section{Introduction}
\label{sect-intro}

Since the early 80s, support varieties have been an important tool in representation theory, when cohomological support and rank support were introduced to provide geometric techniques for studying categories of algebraic origin \cite{AS1982,Carlson1983,Alperin1986}; analogous versions of support in commutative algebra and algebraic geometry were being developed around the same period \cite{Hopkins1987,Avramov1989,NeemanChrom1992}. A {\it{support theory}} here and throughout means a map which assigns to each object of some category a closed subset of some topological space, in a way which is compatible with the structure of the category. By the early 2000s, rank varieties were generalized to finite group schemes via the $\pi$-points of Friedlander--Pevtsova \cite{FP2007}. With the advent of tensor-triangular geometry, a new support theory (the Balmer support, universal support, or tensor-triangular support) based on the Balmer spectrum entered the literature \cite{Balmer2005}. It has found widespread applications in algebraic topology, commutative algebra, algebraic geometry, modular representation theory, equivariant stable homotopy theory, and motivic homotopy theory (see \cite{Balmer2020} and the references therein). Noncommutative Balmer spectra were developed in \cite{NVY2022} to study monoidal triangulated categories without braidings, which arise naturally from various  settings such as stable categories of finite tensor categories (e.g.~ the representation theory of finite-dimensional Hopf algebras), group actions on tensor-triangulated categories via the crossed product (see e.g. \cite{HuangVashaw2025}), and derived categories of bimodules over a noncommutative ring (see \cite[Section 1]{BKSS2020}). 

To greater and lesser extents, these theories of support varieties were developed to understand ``sufficiently small" tensor-triangulated categories, that is, categories of compact objects. From the beginning (e.g.~ when cohomological and rank support were used by Benson--Carlson--Rickard in the 90s to classify the thick ideals of the stable category of finite-dimensional modules for a finite group in characteristic $p$ \cite{BCR1997}), it has been clear that even to understand these ``small" categories, it is sometimes necessary to extend a support theory to arbitrary ``big" objects. For example, in the finite group case, this means extending to infinite-dimensional representations. Fundamentally, this is because Brown representability guarantees that functors preserving set-indexed coproducts have right adjoints \cite[Corollary 2.13]{BIK2012}; but this construction only works when the category is compactly-generated, and in particular requires the existence of arbitrary set-indexed coproducts.

The first extended support varieties, designed to study infinitely-generated representations of a finite group over a field of positive characteristic, were introduced by Benson--Carlson--Rickard in \cite{BCR1995,BCR1996}, which used the Rickard idempotent functors developed in \cite{Rickard1997}. This extended support variety theory was an extension of the rank / cohomological support. This approach was generalized by Benson--Iyengar--Krause to any triangulated category with an action by a graded-commutative Noetherian ring, which plays the role of the cohomology ring \cite{BIK2008, BIK2012}. Shortly after the introduction of the Balmer support for symmetric tensor-triangulated categories, an analogous extension, also using versions of Rickard idempotents, was produced by Balmer--Favi \cite{BF2011}. This version of support was further developed and extended to tensor-triangulated module categories by Stevenson \cite{Stevenson2013, Stevenson2018}. Beyond its use in classifying thick ideals, extended supports have also been used in the wide-ranging program of stratification, i.e.~ in classifying localizing subcategories \cite{BIK2011a, BIK2011b,  Stevenson2013,Stevenson2014b, Stevenson2017local,BIKP2017,BIKP2018,BHS2023a,BHS2023b}.

We are interested in developing the analogous theory in the context of stable categories of finite tensor categories. Finite tensor categories are abelian monoidal categories satisfying finiteness properties; they have played key roles in representation theory, mathematical physics, and topology (see \cite{EGNO2015}). Their stable categories are (possibly noncommutative) monoidal triangulated categories, and there is significant attention currently being given to support varieties, and in particular in producing a concrete support datum that realizes their Balmer support, see \cite{BPW2021, NVY2022,NVY2,NP2023a,NP2023b, NVY3}. While a version of extended support for the noncommutative Balmer support was given in \cite[Appendix B]{NVY3}, there has not yet been an attempt to systematically develop extended support based on Rickard idempotents for supports other than the Balmer one. Developing such a theory would be advantageous, by a theorem of the second author with Nakano and Yakimov, which states that if $\sigma$ is a support map on a monoidal triangulated category $\bK^c$ with value in a Zariski space $X$ which satisfies faithfulness, realization, and a tensor product property, and admits an extension which also satisfies faithfulness and a tensor product property, then $\Spc \bK^c \cong X$ \cite[Theorem 1.4.1]{NVY2022}. We give a precise formulation of this theorem in \Cref{nvy-restated}, after having recalled the necessary technical background in Section \ref{sect-background}.

There is a candidate for a support which realizes the Balmer support for the stable category of a finite tensor category, namely the central cohomological support $\suppC$ introduced by the second author with Nakano and Yakimov \cite[Section 1.4]{NVY3}. This support datum is defined similarly to classical cohomological support, but is based on a particular subalgebra of the cohomology ring, rather than the whole cohomology ring. It is conjectured in \cite[Conjecture E]{NVY3} that the central cohomological support satisfies the conditions of \Cref{nvy-restated}, which would imply that $\Spc \ul{\bC}$ is homeomorphic to the projective spectrum of the ring $C^\bullet$. 

In particular, \cite[Conjecture E]{NVY3} conjectures that the central cohomological support admits a faithful extension which satisfies the generalized tensor product property. The goal of this paper is to develop a general extension for an arbitrary support datum in the noncommutative situation based on Rickard idempotent functors, and determine the precise extent to which the existence of a faithful extension satisfying the tensor product property in the situation of \Cref{nvy-restated} follows from the other assumptions of that theorem. We believe that this represents a foundational question and is a necessary step in developing the theory of  noncommutative support varieties. We first find the following:

\begin{theorem}[\prref{ext-supp-is-ext-supp}, Theorem \ref{comp-faith}, Theorem \ref{tensor-faith-iff}, \Cref{surj-map-faith}]
\label{maintheorem-general}
Let $\bK$ be a rigidly-compactly-generated monoidal triangulated category with compact part $\bK^c$ which is either $\Spc$-Noetherian or has a thick generator. Let $(X,\sigma)$ be a support datum (see Section \ref{sect-background}) defined on $\bK^c$, and let $\sigma$ be faithful (that is, $\sigma(A)=\varnothing$ if and only if $A \cong 0$).
\begin{enumerate}
    \item If $\sigma$ is tensorial ($\bigcup_{B \in \bK^c} \sigma(A \otimes B \otimes C) =\sigma(A) \cap \sigma(B)$ for all $A$ and $B \in \bK^c$) and $X$ is a Zariski space (Noetherian and sober), then $\sigma$ admits an extension $\widetilde{\sigma}$ to $\bK$ via Rickard idempotent functors. Suppose $\sigma$ is comparative, that is,  the universal map $\eta: X \to \Spc \bK^c$ admits a left inverse $\rho$ such that $\eta \rho(\bP) \subseteq \bP$ for all $\bP \in \Spc \bK^c$. Then the extension $\widetilde{\sigma}$ is faithful. 
    \item If $\sigma$ is tensorial, every closed subset $V$ of $X$ has the form $\sigma(A)$ for some $A \in \bK^c$, and $\Spc \bK^c$ is Noetherian, then $\sigma$ admits an extension via Rickard idempotent functors; this extension is faithful if and only if for each $\bP \in \Spc \bK^c$, the closed set $\eta^{-1}(\overline{\{\bP\}})$ in $X$ has a unique generic point.
    \item If $\Spc \bK^c$ is Noetherian and if there exists a closed surjective map $\rho: \Spc \bK^c \to X$ such that $\sigma(A)=\rho(\suppB(A))$ for all $A \in \bK^c$ (where $\suppB$ denotes the Balmer support) then $\sigma$ admits a faithful extension $\widetilde{\sigma}$ to $\bK$ via Rickard idempotent functors.    
\end{enumerate}
\eth

In the special case that $\sigma$ is the universal Balmer support and the monoidal product on $\bK$ is symmetric, the fact that $\widetilde{\sigma}$ is an extended support appears in \cite[Proposition 7.17]{BF2011} and the fact that the extension of $\sigma$ is faithful appears in \cite[Theorem 4.19(2)]{Stevenson2013} and \cite[Theorem 3.22]{BHS2023b}.

In the case of central cohomological support for finite tensor categories, we are able to use \Cref{maintheorem-general} in combination with results from \cite{NVY3} to prove:

\begin{theorem}[\Cref{stablecat-summary}]
    \label{maintheorem-central}
Let $\bC$ be a finite tensor category with stable category $\ul{\bC}$ and categorical center $C^\bullet$ such that $\Ext^\bullet_{\bC}(A,A)$ is finitely-generated as a $C^\bullet$-module, for each $A \in \bC$. Then the central cohomological support $\suppC$ admits a faithful extension to $\ul{\Ind(\bC)}$ if either of the following conditions holds:
\begin{enumerate}
    \item $\Proj C^\bullet$ is Noetherian and central cohomological support is tensorial;
    \item $\Spc \ul{\bC}$ is Noetherian and $C^\bullet$ is finitely-generated.
\end{enumerate}\end{theorem}

This shows that one of the facets of \cite[Conjecture E]{NVY3} follows from the other conjectured results. 

The paper is organized as follows: in \Cref{sect-background} we review background on monoidal triangular geometry and support varieties. In \Cref{sect-extension}, we define extended support via Rickard idempotents in the noncommutative case, based on the classical commutative constructions. In \Cref{sect:tensor-of-balmer}, we prove that the extension of the Balmer support satisfies a generalized tensor prodouct property, and along the way give some useful descriptions of the extended Balmer supports of the Rickard idempotents themselves. In \Cref{sect-extcomp}, we prove that any extension of a support equipped with a comparison map (that is, a map of topological spaces with compatibility properties that allow a ``comparison" between the support in question and the Balmer support) is faithful, as long as either the Balmer spectrum of $\bK^c$ is Noetherian or $\bK^c$ is has a thick generator. As a consequence, we show that our extension of the Balmer support is the unique extension of the Balmer support which is both faithful and satisfies the generalized tensor product property. In \Cref{sect-tensorial}, we show that if the Balmer spectrum is Noetherian, then any faithful and realizing support which satisfies a generalized tensor product property and admits a faithful extension also admits a comparison map. In \Cref{sect-surj}, we prove that if a support map is induced by a surjective map from a Noetherian Balmer spectrum, then the extension of that support is faithful. Lastly, in \Cref{sect-centralcoh}, we turn our attention to the specific case of stable categories of finite tensor categories, and prove \Cref{maintheorem-central}, that is, that extended central cohomological support is faithful under certain assumptions. We then apply these results to answer a question of Pevtsova--Witherspoon \cite{pevtsova2015tensor}, which recovers using different techniques a theorem of Negron--Pevtsova \cite{NP2023a}. 

\begin{remark}[Note on further applications]
\rm{Beyond the applications to finite tensor categories presented in Section \ref{sect-centralcoh}, the results of this paper are applied in joint work of the second author with Nakano and Yakimov \cite{NVY5} to study stratification in monoidal triangular geometry and prove a version of Balmer's Nerves of Steel Conjecture for crossed product categories and in joint work between the second author with Solberg and Witherspoon \cite{SVW} to study Balmer spectra of some monoidal triangulated categories related to bimodules for self-injective algebras.}
\end{remark}

\subsection*{Acknowledgements}
We thank Greg Stevenson and the anonymous referee for helpful suggestions which improved the paper. The first author was partially supported by NSF graduate student fellowship DGE-2140743.
The second author was partially supported by NSF postdoctoral fellowship DMS-2103272 and by an AMS-Simons Travel Grant.

\section{Monoidal triangular geometry and support theories}
\label{sect-background}

In this section, we review the basic background of monoidal triangular geometry. For additional details on classical (commutative) tensor-triangular geometry see \cite{Balmer2005}. For details about the noncommutative generalization, see \cite{NVY2022}.

Let $\bK$ be a triangulated category, with shift functor $\Sigma$. For background on triangulated categories, see \cite{Neeman2001}. Recall that a triangulated subcategory is one closed under cones and $\Sigma^{\pm}$; a thick subcategory is a full triangulated subcategory closed under direct summands; a localizing subcategory is a full thick subcategory closed under arbitrary set-indexed coproducts (if such coproducts exist). We briefly remark that thickness follows automatically from the other properties of localizing subcategories as long as $\bK$ is closed under set-indexed coproducts via
the Eilenberg swindle, but the thickness hypothesis is important if such coproducts do not exist.

Throughout, we assume that $\bK$ is a rigidly-compactly-generated monoidal triangulated category. To spell this out:
\begin{enumerate}
    \item $\bK$ is triangulated, and contains all set-indexed coproducts;
    \item $\bK$ is generated, as a localizing subcategory, by $\bK^c$, the subcategory of $\bK$ consisting of compact objects (that is, $\bK$ is the smallest localizing subcategory of $\bK$ which contains $\bK^c$);
    \item $\bK$ is equipped with a biexact tensor product $- \otimes -$ with tensor unit $\unit$ satisfying the usual associativity and unit axioms \cite[Definition 2.1.1]{EGNO2015}, and the tensor product commutes with arbitrary set-indexed coproducts;
    \item $\unit \in \bK^c$, and if $A$ and $B$ are in $\bK^c$, then so is $A \otimes B$ (in particular, $\bK^c$ is a tensor subcategory of $\bK$);
    \item $\bK^c$ is rigid, as a monoidal category (that is, every compact object has a left and right dual, \cite[Definition 2.10.1, 2.10.2]{EGNO2015}).
\end{enumerate}
Note that we do not assume the tensor product is symmetric or braided. Our primary motivating example, which we explore in more detail in Section \ref{sect-centralcoh}, is the case of stable categories of finite tensor categories. 

\bde{bigsmall}    We call a monoidal triangulated category {\it{big}} if it satisfies (1)-(5) above, that is, if it is a rigidly-compactly-generated monoidal triangulated category. We call a monoidal triangulated category {\it{small}} if it is equivalent to $\bK^c$ for some big monoidal triangulated category $\bK$. We call $\bK^c$ {\it{finitely-generated}} if there exists an object $G$ in $\bK^c$ such that the smallest thick subcategory containing $G$ is $\bK^c$. 
\ede

\bre{all-essentially-small-are-small}
By the well-developed theory of ind-completions of categories (see e.g.~ \cite{Krause2024} and the references contained therein), basically every naturally-occurring essentially small triangulated category will be small in the sense of \deref{bigsmall}.
\ere

Recall that a {\it thick ideal} of $\bK^c$ is a thick subcategory closed under tensor products with arbitrary objects of $\bK^c$. Given a collection of objects $\bS$ of $\bK^c$, we will denote $\langle \bS \rangle$ the smallest thick ideal containing $\bS$. Likewise, we denote by $\Loc(\bS)$ and $\Loc^{\otimes}(\bS)$ the smallest localizing subcategory and smallest localizing ideal of $\bK$ generated by any collection of objects $\bS$ in $\bK$. A proper thick ideal $\bP$ is called {\it prime} if 
\[
\bI \otimes \bJ \subseteq \bP \Rightarrow \bI \text{ or } \bJ \subseteq \bP
\]
for all thick ideals $\bI$ and $\bJ$. The category $\bK^c$ is called {\it{completely prime}} if every prime ideal is completely prime, that is, for a prime ideal $\bP$, we have $A \otimes B$ in $\bP$ implies $A$ or $B$ in $\bP$, for any objects $A$ and $B$ in $\bK^c$. The {\it{Balmer spectrum}} $\Spc \bK^c$ is defined as the topological space consisting of prime thick ideals of $\bK^c$, with closed sets
\[
\suppB(\bS) := \{ \bP \in \Spc \bK^c : \bP \cap \bS = \varnothing\},
\]
for any collection $\bS$ of objects of $\bK^c$. Appropriately, we call $\bK^c$ {\it $\Spc$-Noetherian} if $\Spc \bK^c$ is Noetherian. The Balmer spectrum is useful conceptually because under mild conditions, it gives a topological classification of the thick ideals of $\bK^c$ by Thomason subsets of $\Spc \bK^c$. In general, for a topological space $X$, $\Thom(X)$ will denote the set of Thomason subsets of $X$ (recall that a set is Thomason if it is a union of closed sets, each of which has a quasicompact complement).
The classification theorem is given by the following:

\begin{theorem}
\label{class-ideals}
When $\bK^c$ is commutative, completely prime, $\Spc$-Noetherian, or finitely-generated, there are inverse bijections between the thick ideals of $\bK^c$ and the Thomason subsets of $\Spc \bK^c$ given by
\begin{center}
\begin{tikzcd}
\bI \arrow[rr, maps to]                &  & \Phi_{\suppB}(\bI)  := \bigcup_{A \in \bI} \suppB(A)          \\
&&\\
\Ideals(\bK^c) \arrow[rr, bend left]   &  & \Thom(\Spc \bK^c) \arrow[ll, bend left] \\
&&\\
\Theta_{\suppB}(S) :=\{A \in \bK^c: \suppB(A) \subseteq S\} &  & S \arrow[ll, maps to]                  
\end{tikzcd}\end{center}
This bijection restricts to a bijection between principal thick ideals and closed sets of $\Spc \bK^c$ of the form $\suppB(A)$, for some $A \in \bK^c$. When $\Spc \bK^c$ is Noetherian, this restricts to a bijection between principal ideals and closed subsets of $\Spc \bK^c$, that is, every closed subset of $\Spc \bK^c$ is of the form $\suppB(A)$ for some $A \in \bK^c.$
\end{theorem}

\begin{proof}
    See \cite{Balmer2005} for the case that $\bK^c$ is commutative, \cite{MallickRay2023} for the case that $\bK^c$ is completely prime, \cite{Rowe} for the case that $\bK^c$ is $\Spc$-Noetherian, and \cite{NVY4} for the case that $\bK^c$ is finitely-generated (i.e., has a thick generator). 
\end{proof}

See also \cite{DDM} for a comprehensive discussion on thick ideal classification outside the contexts of Theorem \ref{class-ideals}. 

In practice, it is difficult to determine the Balmer spectrum concretely. To do so, we turn to the theory abstract support varieties. A {\it support datum} $(X,\sigma)$ on $\bK^c$ consists of a map $\sigma$ from objects of $\bK^c$ to closed subsets of some $T_0$ topological space $X$, satisfying 
\begin{enumerate}
    \item $\sigma(\unit)=X, \sigma(0)=\varnothing$;
    \item $\sigma(A \oplus B)= \sigma(A) \cup \sigma(B)$ for all $A, B \in \bK^c$;
    \item $\sigma(\Sigma A) = \sigma(A)$ for all $A \in \bK^c;$
    \item if $A \to B \to C \to \Sigma A$ is a distinguished triangle in $\bK^c$, then $\sigma(A) \subseteq \sigma(B) \cup \sigma(C);$
    \item $\sigma(A \otimes B) \subseteq \sigma(A) \cap \sigma(B)$ for all $A,B \in \bK^c$.
\end{enumerate}
When $X$ is clear from context (or if the particular topological space $X$ is immaterial to the discussion at hand), we will refer to $\sigma$ as the support datum, or, alternatively, as the support map. We will call a support map $\sigma$ {\it tensorial} if it satisfies
\begin{enumerate}
    \item[(6)] $\bigcup_{B \in \bK^c} \sigma(A \otimes B \otimes C) = \sigma(A) \cap \sigma(C)$, for all $A$ and $C$ in $\bK^c$;
\end{enumerate}
{\it faithful} if
\begin{enumerate}
    \item[(7)] $\sigma(A)= \varnothing \Rightarrow A \cong 0$, for all $A \in \bK^c$; 
\end{enumerate}
and {\it realizing} if
\begin{enumerate}
    \item[(8)] for any closed subset $V$ of $X$, there exists an object $A$ with $\sigma(A)=V$.
\end{enumerate}

\bre{t-zero}
The assumption that the topological space $X$ be $T_0$ is not included in the axioms put forward in \cite{Balmer2005}, but we include it here since it will be necessary throughout the paper. Note that the Balmer spectrum $\Spc \bK^c$ is automatically $T_0$. 
\ere

\bex{balmer-supp}
The pair $(\Spc \bK^c, \suppB)$ is a tensorial and faithful support datum. Recall that we are assuming throughout that $\bK^c$ is rigid; this is necessary for the faithfulness of $\suppB$ (\cite[Proposition 4.1.1]{NVY2}). If $\Spc \bK^c$ is Noetherian, then $\suppB$ is also realizing, by \Cref{class-ideals}. 
\eex

Tensoriality is particularly important, due to the universal property of the Balmer spectrum:

\begin{theorem}[{\cite[Theorem 5.3]{BKS2007}}]
\label{Balmer-univ}
If $(X,\sigma)$ is a tensorial support datum on $\bK^c$, then there is a unique continuous map 
\[
X \xrightarrow{\eta} \Spc \bK^c
\]
such that $\sigma(A)=\eta^{-1} (\suppB(A))$ for all $A \in \bK^c.$ 
\end{theorem}

The following result about the universal map $\eta$ is implied by the proof of \cite[Theorem 6.2.1(b)(ii)]{NVY2022}.

\bpr{eta-inj}
Let $(X,\sigma)$ be a realizing support datum on a monoidal triangulated category $\bK^c$. Then the universal map $\eta: X \to \Spc \bK^c$ as in \Cref{Balmer-univ} is injective. 
\epr

Note this proposition uses the fact that $X$ is assumed to be $T_0$, which is an assumption we make about all support data in this paper. 

\bnota{phi-theta}
Given a support datum $\sigma$, we have analogues of the maps from \Cref{class-ideals}: given a thick ideal $\bI$ of $\bK^c$, and a subset $S$ of $X$, we will denote
\begin{align*}
    \Phi_{\sigma}(\bI) &:= \bigcup_{A \in \bI} \sigma(A), \text{ and} \\
    \Theta_{\sigma}(S)&:=\{A \in \bK^c : \sigma(A) \subseteq S\}.
\end{align*}
When there is no risk of confusion, we will drop the subscript, and write $\Phi$ and $\Theta$ for $\Phi_{\sigma}$ and $\Theta_{\sigma}$, respectively.
\end{notation}

This paper will be primarily concerned with extensions of a support $(X,\sigma)$ to the big category $\bK$. By an extension, we mean the following. 

\bde{ext-arb-supp}
Let $(X,\sigma)$ be a support datum on $\bK^c$. We say that $(X,\widetilde{\sigma})$ is an {\it extension} of $(X,\sigma)$ if $\widetilde{\sigma}$ is a map from $\bK$ to (not necessarily closed) subsets of $X$, and
\begin{enumerate}
    \item $\widetilde{\sigma}(A)=\sigma(A)$ for all $A$ compact;
    \item $\widetilde{\sigma}(\coprod_{i \in I} A_i)= \bigcup_{i\in I}\widetilde{\sigma}(A_i)$ for any set-indexed collection $\{A_i : i \in I\}$ of objects of $\bK$;
    \item $\widetilde{\sigma}(\Sigma A) = \widetilde{\sigma}(A)$ for all $A \in \bK;$
    \item if $A \to B \to C \to \Sigma A$ is a distinguished triangle in $\bK$, then $\widetilde{\sigma}(A) \subseteq \widetilde{\sigma}(B) \cup \widetilde{\sigma}(C);$
    \item $\widetilde{\sigma}(A \otimes B) \subseteq \widetilde{\sigma}(A) \cap \widetilde{\sigma}(B)$ for all $A,B \in \bK$.
\end{enumerate}
We call $\widetilde{\sigma}$ {\it tensorial} if
\begin{enumerate}
    \item[(6)]  $\bigcup_{B \in \bK^c} \widetilde{\sigma}(A \otimes B \otimes C) = \widetilde{\sigma}(A) \cap \widetilde{\sigma}(C)$, for all $A$ in $\bK$ and $C \in \bK^c$,
\end{enumerate}
and {\it faithful} if
\begin{enumerate}
    \item[(7)] $\widetilde{\sigma}(A)= \varnothing \Rightarrow A \cong 0$, for all $A \in \bK$. 
\end{enumerate}
\ede

The following theorem is our primary motivation in providing an extension for a support: it is a necessary condition, in many cases, for providing concrete realizations of the Balmer spectrum and Balmer support.

\begin{theorem}
[{\cite[Theorem 1.4.1]{NVY2022}}]
\label{nvy-restated}
Let $(X,\sigma)$ be a tensorial, faithful, realizing support with value in a Zariski space $X$ on a small monoidal triangulated category $\bK^c$. If $(X,\sigma)$ admits a tensorial, faithful extension $(X,\widetilde{\sigma})$, then there is a homeomorphism 
\[
\eta: X \xrightarrow{\sim} \Spc \bK^c
\]
such that $\eta^{-1}(\suppB(A)) = \sigma(A)$ for all $A \in \bK^c$. 
\end{theorem}

For reference later, we note the elementary fact that  extended supports are compatible with taking localizing subcategories in the following way:

\ble{support-local}
Let $\sigma$ be a support datum on $\bK^c$ with extension $\widetilde{\sigma}$. Let $\bI$ be a thick ideal of $\bK^c$, and let $\Loc(\bI)$ denote the smallest localizing subcategory of $\bK$ containing $\bI$. If $A$ is in $\Loc(\bI)$, then $\widetilde{\sigma}(A) \subseteq \Phi_{\sigma} (\bI).$ Furthermore, if $B \in \bK$ and $C$ is in the localizing tensor ideal generated by $B$, then $\widetilde{\sigma}(C) \subseteq \widetilde{\sigma}(B)$.
\ele

\begin{proof}
    Using the axioms of extension, it is clear that the collection of objects
    \[
    \bL := \{ D \in \bK : \widetilde{\sigma}(D) \subseteq \Phi_{\sigma}(\bI)\}
    \]
    is a localizing subcategory. Since $\bI \subseteq \bL$ by the fact that $\sigma = \widetilde{\sigma}$ on compact objects, it follows that $\Loc(\bI) \subseteq \bL$. The second claim is similar.
\end{proof}

\section{Rickard idempotents and formal extensions of support}
\label{sect-extension}

In this section we review background on Rickard idempotent functors, and subsequently construct formal extensions of support theories.

Let $\bK$ be a big monoidal triangulated category, and let $\bI$ be a thick ideal of $\bK^c$. Then the Verdier quotient functor
\[
\bK \to \bK / \Loc(\bI)
\]
admits a right adjoint, by a general argument of Neeman (see \cite[Lemma 1.7]{Neeman1992}) which was inspired by older arguments of Adams and Bousfield. By composing the quotient functor with its adjoint, one obtains a functor 
$L_{\bI}: \bK \to \bK$ which takes value in $\Loc(\bI)^{\perp}$, the subcategory of $\bK$ of objects $B$ such that $\Hom_{\bK}(A,B)=0$ for all $A \in \Loc(\bI)$. By analogous argument, there exists a right adjoint to the inclusion functor
\[
\Loc(\bI) \to \bK,
\] and again by composition we get a functor $\it{\Gamma}_{\bI}: \bK \to \bK$ which takes value in $\Loc(\bI)$; see also \cite[\S7.2]{Krause2010}. A similar construction was given by Miller in \cite{Miller1992} for the category of spectra. Given an object $A$, the object $L_{\bI}A$ can be described directly by a process of repeated homotopy colimits given by ranging over all possible maps from objects of $\bI$. This construction was also used directly by Rickard in \cite{Rickard1997}, who subsequently used the functors $\it{\Gamma}_{\bI}$ and $L_{\bI}$ in joint work with Benson and Carlson to define the first extended supports in modular representation theory \cite{BCR1995,BCR1996}. The functors $\it{\Gamma}_{\bI}$ and $L_{\bI}$, which are now commonly referred to as {\it{Rickard idempotent functors}}, satisfy the following universal property:

\begin{theorem}[\cite{Neeman1992,Rickard1997}]\label{tensor-idem}
    For any thick subcategory $\bI$ of $\bK^c$, there exist unique functors $L_{\bI}, \itGamm_{\bI}: \bK \to \bK$ such that for each $A \in \bK$, there is a distinguished triangle
    \[
    \itGamm_{\bI} A \to A \to L_{\bI}A \to \Sigma \itGamm_{\bI} A
    \]
    with $\itGamm_{\bI}A \in \Loc(\bI)$ and $L_{\bI} A \in \Loc(\bI)^\perp.$
\end{theorem}

We summarize some basic properties satisfied by the Rickard idempotent functors which we will use throughout.

\begin{lemma}
\label{lemma:commutes_with_Gamma_L}
Let $\bI$ and $\bJ$ be thick ideals of $\bK^c$ such that $\bJ \subseteq \bI$, and let $A \in \bK$. Then:
\begin{enumerate}
    \item $\itGamm_{\bI}A \cong A $ if and only if $L_{\bI} A \cong 0$ if and only if $A \in \Loc(\bI)$;
    \item $A \otimes \itGamm_{\bI} \unit \cong \itGamm_{\bI}A \cong \itGamm_{\bI} \unit \otimes A;$
    \item $A \otimes L_{\bI} \unit \cong L_{\bI}A \cong L_{\bI} \unit \otimes A;$
    \item $L_{\bI} \unit \otimes \itGamm_{\bI} \unit \cong 0;$
    \item $\itGamm_{\bI} \unit \otimes \itGamm_{\bJ} \unit \cong \itGamm_{\bJ} \unit$ and $L_{\bI} \unit \otimes L_{\bJ} \unit \cong L_{\bI} \unit;$
    \item there is a unique morphism of triangles
    \begin{center}
        \begin{tikzcd}
\itGamm_{\bJ}\unit \arrow[d] \arrow[r] & \unit \arrow[r] \arrow[d, no head, Rightarrow, no head] & L_{\bJ} \unit \arrow[r] \arrow[d] & {} \\
\itGamm_{\bI} \unit \arrow[r]          & \unit \arrow[r]                                & L_{\bI} \unit \arrow[r]           & {}
\end{tikzcd}
    \end{center}
    (that is, the horizontal maps are the triangles given by Theorem \ref{tensor-idem}, and the maps $\itGamm_{\bJ} \unit\to \itGamm_{\bI} \unit$ and $L_{\bJ}  \unit \to L_{\bI} \unit$ are the unique maps making this diagram commute).
\end{enumerate}
\end{lemma}

\begin{proof}
(1) is clear. (2) and (3) follow from \cite[Lemma 3.2.1]{NVY4}. (4) is a direct consequence of (1), (2), and the fact that $\itGamm_{\bI} \unit \in \Loc (\bI)$. (5) is straightforward. For (6), see \cite[Remark 2.9]{BF2011}.\end{proof}

Note in particular that by Lemma \ref{lemma:commutes_with_Gamma_L}(2) and (3), both $L_{\bI} \unit$ and $\itGamm_{\bI} \unit$ tensor-commute with all objects of $\bK$. This fact will be used frequently throughout the paper. 

\bnota{notation-gamma-closed-set}
Recall that given a support datum $(X,\sigma)$ and a specialization-closed subset $S$ of $X$, we have the thick ideal $\Theta(S)$ defined in \notaref{phi-theta}. To avoid burdening the notation, we set
\[
\itGamm_{S} := \itGamm_{\Theta(S)}, \; \; L_{S} := L_{\Theta(S)}.
\]
\enota

We now define extended support for an arbitrary support datum $(X,\sigma)$; this follows the definitions put forward by Benson--Iyengar--Krause \cite{BIK2008} and Balmer--Favi \cite{BF2011} from the commutative case.

Let $x$ be a point of $X$. We set 
\begin{align*}
\mathcal{Z}(x)&:=\{y \in X : x \not \in \overline{\{y\}} \},\\
\mathcal{V}(x) &:= \overline{\{x\}}.
\end{align*}
Note that $\mc{Z}(x)$ and $\mc{V}(x)$ are both specialization-closed, and, since we assume that $X$ is $T_0$, we have $\{x\} = \mc{V}(x) \backslash (\mc{V}(x) \cap \mc{Z}(x))$. We now define a functor $\itGamm_x$ via 
\[
\itGamm_x A := \itGamm_{\mathcal{V}(x)}\unit \otimes L_{\mathcal{Z}(x)}\unit\otimes A.
\]
Note that by Lemma \ref{lemma:commutes_with_Gamma_L},  $\itGamm_x$ is the composition of functors $\itGamm_{\mathcal{V}(x)}\circ L_{\mathcal{Z}(x)}$.
\bre{thomason-not-spec-closed}
The results in this paper will focus on support data based on Noetherian spaces $X$. For non-Noetherian spaces, one can define $\it{\Gamma}_x$ in a more general way using Thomason subsets, as discussed in \cite[\S2.1]{stevenson2014a}, to get around the issue that $\mc{V}(x)$ might not be Thomason.
Recall that in a Noetherian space, any specialization closed subset is automatically a Thomason subset.
\ere

\bre{confusing-notation}
We caution the reader that if $x$ is a point, then $\itGamm_x$ is not the same as $\itGamm_{\overline{\{x\}}}$, which is defined according to \notaref{notation-gamma-closed-set}.
\ere

\bnota{confusing-notation-two}
Unfortunately, the notation $\itGamm_{\bP}$, for $\bP \in \Spc \bK^c$, is now overloaded: (1) if we consider $\bP$ as a thick ideal, it has a meaning as in \Cref{tensor-idem}, and, on the other hand, (2) if we consider $\bP$ as a point of $\Spc \bK^c$ equipped with the Balmer support $\suppB$, then we have the functor
\[
\itGamm_{\bP}:=\itGamm_{\mc{V}(\bP)} \circ L_{\mc{Z}(\bP)} = \itGamm_{\mc{V}(\bP)}\unit \otimes L_{\mc{Z}(\bP)}\unit\otimes -
\]
as defined above. For the remainder of the paper, if we write $\itGamm_{\bP}$ for $\bP \in \Spc \bK^c$, we will always mean (2), that is, we consider $\bP$ as a point of $\Spc \bK^c$ rather than as a thick ideal. 
\enota

\bde{ext-support}
Let $(X,\sigma)$ be a support datum on $\bK^c$. Define the {\it{extension of $\sigma$}}, denoted $(X, \widetilde{\sigma})$, by 
\[
\widetilde{\sigma}(A):=\{x \in X : \itGamm_x A \not \cong 0\}
\]
for all $A \in \bK.$
\ede

It is straightforward to check the following:

\bpr{ext-supp-is-ext-supp}
Let $(X,\sigma)$ be a tensorial realizing support on a small monoidal triangulated category $\bK^c$. Then $(X, \widetilde{\sigma})$ is an extension of $(X,\sigma)$ (recall \deref{ext-arb-supp}). 
\epr

\begin{proof}
    Since both functors $\itGamm_{\bI}$ and $L_{\bI}$ are triangulated functors commuting with coproducts, conditions (2)-(4) of \deref{ext-arb-supp} are automatic. Condition (5) follows from the fact that $F(A\otimes B)\cong F(A)\otimes F(B)$ for both of the functors $F=\itGamm_{\bI},L_{\bI}$ (these identities follow directly from Lemma~\ref{lemma:commutes_with_Gamma_L}). It remains to show condition (1), i.e. that $\sigma(A)=\widetilde{\sigma}(A)$ for all compact objects $A$; the proof of this fact will proceed similarly to \cite[Proposition~7.17]{BF2011}, which proves the analogous result for $\sigma = \suppB$ in the commutative case. We need to show that for any compact $A$, we have $x \not \in \sigma(A)$ if and only if $\itGamm_x \unit \otimes A \cong 0$. 

    Suppose $x \not \in \sigma(A)$; this is equivalent to $\sigma(A) \subseteq \mc{Z}(x)$. That in turn is equivalent to $L_{\mc{Z}(x)} \unit \otimes A \cong 0$ by Lemma \ref{lemma:commutes_with_Gamma_L}(1), which then immediately implies $\itGamm_x \unit \otimes A \cong 0$. 

    On the other hand, assume $\itGamm_x \unit \otimes A \cong 0$. By the assumption that $\sigma$ is realizing, we can find an object $C \in \bK^c$ with $\sigma(C)=\mc{V}(x)$. Note that $C \otimes \itGamm_{\mc{V}(x)} \unit \cong C$. For any $B \in \bK^c$, we then find that
    \begin{align*}
        0 &\cong A \otimes \itGamm_x \unit \otimes B \otimes C\\
        &\cong A \otimes L_{\mc{Z}(x)}\unit\otimes B \otimes C\\
        & \Updownarrow \\
        A \otimes B \otimes C &\in \Loc(\Theta(\mc{Z}(x)))\\
        &\Updownarrow \text{ (by \cite[Lemma 2.2]{Neeman1992})}\\
        A \otimes B \otimes C &\in \Theta(\mc{Z}(x))\\
        &\Updownarrow\\
        \sigma(A \otimes B \otimes C) &\subseteq \mc{Z}(x).
    \end{align*}
Since we assume that $\sigma$ is tensorial, this implies that
\[
\bigcup_{B \in \bK^c} \sigma(A \otimes B \otimes C) = \sigma(A) \cap \sigma(C) \subseteq \mc{Z}(x).
\]
Now note that $x \not \in \mc{Z}(x)$, but $x \in \mc{V}(x)$, and recall that we have assumed $\sigma(C)=\mc{V}(x)$. It follows that $x \not \in \sigma(A)$. 
\end{proof}

\section{Tensoriality of extended Balmer support}
\label{sect:tensor-of-balmer}

In this section, we prove that when $\Spc \bK^c$ is Noetherian, the extension of the Balmer support is tensorial. This will follow closely the proof given in the commutative case by Balmer and Favi \cite[Section 7]{BF2011}.

We first set some notation.

\bnota{balmersupp-not}
    The extension of the Balmer support $\suppB$ as in \deref{ext-support} will be denoted $\SuppB$.
\end{notation}

To prove tensoriality, we first describe the extended Balmer supports of the Rickard idempotents.

\ble{rickard-idem-supp}
Let $\bK^c$ be a small $\Spc$-Noetherian monoidal triangulated category. Let $\bI$ be any thick ideal of $\bK^c$. Then $\SuppB(\itGamm_{\bI} \unit)  = \Phi_{\suppB}(\bI)$ and $\SuppB(L_{\bI} \unit) = \Spc \bK^c \backslash \Phi_{\suppB}(\bI).$
\ele

\begin{proof}
Let $A \in \bI$, and let $\bP \in \suppB(A).$ Then $A \otimes \itGamm_{\bP} \unit  \not \cong 0.$ Since $\itGamm_{\bI} \unit \otimes A \cong A \cong A \otimes \itGamm_{\bI} \unit$ by Lemma \ref{lemma:commutes_with_Gamma_L}(1) and (2), it follows that $\itGamm_{\bI} \unit\otimes \itGamm_{\bP} \unit \not \cong 0,$ that is, $\bP \in \SuppB(\itGamm_{\bI} \unit).$ Hence 
        \[
        \Phi_{\suppB}(\bI) \subseteq \SuppB(\itGamm_{\bI}(A)). 
        \]
        On the other hand, since $\itGamm_{\bI}(A)$ is in $\Loc(\bI)$, we have $\SuppB(\itGamm_{\bI}(A)) \subseteq \Phi_{\suppB}(\bI)$ by \leref{support-local}. 

        Suppose $\bP \in \SuppB(L_{\bI} \unit)$, i.e. $\itGamm_{\bP} \unit \otimes L_{\bI} \unit \not \cong 0.$ Recall that by Lemma \ref{lemma:commutes_with_Gamma_L}, if $\Theta(\mc{V}(\bP)) \subseteq \bI$, this would imply that $\itGamm_{\mc{V}(\bP)} \unit \otimes L_{\bI} \unit \cong 0$, which would immediately imply that $\itGamm_{\bP} \unit \otimes L_{\bI} \unit \cong 0$, a contradiction, so we have $\Theta(\mc{V}(\bP)) \not \subseteq \bI$. Using the Noetherianity of $\Spc \bK^c$, by \Cref{class-ideals} we have 
        \[
        \mc{V}(\bP) \not \subseteq \Phi_{\suppB}(\bI), 
        \]
        i.e. $\bP \not \in \Phi_{\suppB}(\bI)$. It follows that
        \[
        \SuppB(L_{\bI} \unit) \subseteq \Spc \bK^c \backslash \Phi_{\suppB}(\bI).
        \]        
        Since we have a distinguished triangle
        \[
        \itGamm_{\bI} \unit \to \unit \to L_{\bI} \unit,
        \]
        we have that $\SuppB(\itGamm_{\bI} \unit) \cup \SuppB(L_{\bI} \unit) = \Spc \bK^c,$ and so the opposite containment
        \[
        \SuppB(L_{\bI} \unit) \supseteq \Spc \bK^c \backslash \Phi_{\suppB}(\bI)
        \] 
        holds as well.
\end{proof}

We now prove that $\SuppB$ satisfies a tensor product property when tensoring by Rickard idempotents.

\ble{tensor-idem-intersect}
Let $\bK^c$ be a small $\Spc$-Noetherian monoidal triangulated category. Let $A$ be an object of $\bK$, and let $\bI$ be a thick ideal of $\bK^c$. Then
\begin{align*}
\SuppB(A \otimes \itGamm_{\bI}\unit ) &= \SuppB(A) \cap \SuppB(\itGamm_{\bI} \unit)= \SuppB(A) \cap \Phi_{\suppB}(\bI),\\
\SuppB(A \otimes L_{\bI}\unit ) &= \SuppB(A) \cap \SuppB(L_{\bI} \unit)=\SuppB(A) \cap (\Spc \bK^c \backslash \Phi_{\suppB}(\bI)).
\end{align*}
\ele

\begin{proof}
Since $\SuppB$ is an extended support, the containments 
\begin{align*}
\SuppB(A \otimes \itGamm_{\bI}\unit ) &\subseteq \SuppB(A) \cap \SuppB(\itGamm_{\bI} \unit),\\
\SuppB(A \otimes L_{\bI}\unit ) &\subseteq \SuppB(A) \cap \SuppB(L_{\bI} \unit)
\end{align*}
are automatic. For the opposite containment in the first claim, suppose $\bP \in \SuppB(A) \cap \SuppB(\itGamm_{\bI} \unit)$. In particular, by \leref{rickard-idem-supp}, $\bP \in \Phi_{\suppB}(\bI)$, in which case $\mc{V}(\bP) \subseteq \Phi_{\suppB}(\bI)$, and hence $\Theta(\mc{V}(\bP)) \subseteq \bI$. It follows from Lemma \ref{lemma:commutes_with_Gamma_L}(5) that $\itGamm_{\mc{V}(\bP)} \unit \otimes \itGamm_{\bI} \unit \cong \itGamm_{\mc{V}(\bP)} \unit$, and so
\begin{align*}
    A \otimes \itGamm_{\bI} \unit \otimes \itGamm_{\bP} \unit & \cong A \otimes \itGamm_{\bI} \unit \otimes \itGamm_{\mc{V}(\bP)} \unit \otimes L_{\mc{Z}(\bP)} \unit \\
    & \cong A \otimes \itGamm_{\mc{V}(\bP)} \unit \otimes L_{\mc{Z}(\bP)} \unit\\
    &\cong A \otimes \itGamm_{\bP} \unit\\
    & \not \cong 0,
\end{align*}
so $\bP \in \SuppB(A \otimes \itGamm_{\bI} \unit)$. The second claim is similar.
\end{proof}

We are now able to prove tensoriality of the Balmer spectrum by adapting the proof from the commutative case, see \cite[Theorem 7.22]{BF2011}.

\bpr{tensor-balmer}
Let $\bK^c$ be a small $\Spc$-Noetherian monoidal triangulated category. Then the extended Balmer support $\SuppB$ is tensorial. \epr

\begin{proof}

Let $A \in \bK$ and $C \in \bK^c$. Since $\SuppB$ is an extended support, it is automatic that
\[
\bigcup_{B \in \bK^c} \SuppB(A \otimes B \otimes C) \subseteq \SuppB(A) \cap \SuppB(C).
\]
We must show the opposite containment. Set $V:=\suppB(C)$. We claim that $C$ and $\itGamm_V \unit$ generate the same localizing tensor ideal. Indeed, $\itGamm_V \unit \otimes C \cong C$ by Lemma \ref{lemma:commutes_with_Gamma_L} since $C \in \Theta(V)$, so $C$ is in the localizing tensor ideal generated by $\itGamm_V \unit$; and on the other hand, $\itGamm_V \unit$ is by definition in $\Loc(\Theta(V))$, which is equal to $\Loc(\langle C \rangle)$ by \Cref{class-ideals}, so $\itGamm_V \unit$ is in the localizing tensor ideal generated by $C$. Since the kernel of $\itGamm_{\bP} \otimes -$ is a localizing tensor ideal for every $\bP \in \Spc \bK^c$, it follows that $\SuppB(C) = \SuppB(\itGamm_V \unit)$. Hence
\[
\SuppB(A) \cap \SuppB(C) = \SuppB(A) \cap \SuppB(\itGamm_V \unit) = \SuppB(A \otimes \itGamm_V \unit)
\]
by \leref{tensor-idem-intersect}. Suppose $\bP \in \SuppB(A \otimes \itGamm_V \unit)$, so that $\itGamm_{\bP}\unit \otimes A \otimes \itGamm_V \unit \not \cong 0$. Since $\itGamm_V \unit \in \Loc(\langle C \rangle)$, it is straightforward to check that there must be some $D \in \langle C \rangle$ with $\itGamm_{\bP} \unit \otimes A \otimes D \not \cong 0$. This implies there is some $B \in \bK^c$ with $\itGamm_{\bP} \unit \otimes A \otimes B \otimes C \not \cong 0$ by the following argument. Suppose $\itGamm_{\bP} \unit \otimes A \otimes B \otimes C \cong 0$ for all $B \in \bK^c$. Consider
\[
\bJ:=\{ E \in \bK^c: \;\; \itGamm_{\bP} \unit \otimes A \otimes B \otimes E \cong 0 \; \forall \;B \in \bK^c \}.
\]
It can be readily checked that $\bJ$ is a thick ideal, and since it contains $C$, it must contain all of $\langle C \rangle$, a contradiction. Hence there exists $B$ with $\itGamm_{\bP} \unit \otimes A \otimes B \otimes C \not \cong 0,$ and we have proven the reverse containment.

\end{proof}

\section{Faithfulness of extended comparative support}
\label{sect-extcomp}

Throughout this section, we will assume that 
\begin{equation}
\label{kc-condition}
    \bK^c \text{ is either }\Spc\text{-Noetherian or finitely-generated,}
\end{equation}
the reason being that either allows us to apply Theorem \ref{class-ideals}. We will assume that $(X,\sigma)$ is a tensorial support such that there exists a map $\rho: \Spc \bK^c \to X$ satisfying
\begin{enumerate}
    \item $\rho \circ \eta = \id_X$, and
    \item $(\eta \circ \rho) (\bP) \subseteq \bP$ for all $\bP \in \Spc \bK^c$, 
\end{enumerate}
where $\eta: X \to \Spc \bK^c$ is the universal map that exists by \Cref{Balmer-univ}. Recall that by the definition of the topology on $\Spc \bK^c$, (2) is equivalent to the condition that $\eta \rho(\bP) \in \overline{\{\bP\}}.$

\bde{comparative}    A support datum $(X,\sigma)$ is called {\it{comparative}} if it satisfies (1) and (2) above for some map $\rho$. The map $\rho$ will be called a {\it{comparison map}}.
\ede

The terminology is intended to reflect that $\rho$  is an analogue of the comparison map (also denoted $\rho$) produced by Balmer in \cite[Theorem 5.3]{Balmer2010}. In this section, we will show that the extension $\widetilde{\sigma}$ of a comparative support $\sigma$ to $\bK^c$ is faithful, as long as $X$ is a Zariski space and $\bK^c$ is finitely-generated (recall, this means that there exists an object $G$ in $\bK^c$ such that $\bK^c$ is the smallest thick subcategory of $\bK^c$ containing $G$) or $\Spc$-Noetherian. Recall here that a sober topological space is one in which every irreducible closed set has a unique generic point, and a Zariski space is one which is Noetherian and sober. The main applications of this theorem are twofold: (1) for $\sigma = \suppB$, which is manifestly comparative (both $\rho$ and $\eta$ are the identity maps), and (2) the central cohomological support for the stable category of a finite tensor category, a version of support introduced in \cite{NVY3}, which we explore in Section \ref{sect-centralcoh}. 

We begin by noting a few elementary properties of comparative support that will be useful.

\begin{lemma}
    \label{comparative-props}
    Let $(X,\sigma)$ be a comparative support on $\bK^c$, and let $X$ be sober. Then 
    \begin{enumerate}
        \item if $W$ is a closed subset of $\Spc \bK^c$, then $\rho(W)=\eta^{-1}(W)$;
        \item $\sigma(A)= \rho(\suppB(A))$ for any $A \in \bK^c$;
        \item $\sigma$ is faithful;
        \item if $V \subseteq X$ is a Thomason subset, then $\rho^{-1}(V)$ is Thomason in $\Spc \bK^c$;
        \item if $X$ is Noetherian and $\bK^c$ is finitely-generated, then $\sigma$ is realizing.
    \end{enumerate}
    \end{lemma}

\begin{proof}
(1): let $W$ be a closed subset of $\Spc \bK^c$. If $x \in \eta^{-1}(W)$, then $\eta(x) \in W$, and $\rho \eta(x) = x \in \rho(W).$ For the other direction, if $\bP \in W$, that is $\rho(\bP) \in \rho(W)$, then $\eta\rho(\bP) \subseteq \bP$ and hence $\eta \rho(\bP) \in W$, since $W$ is closed.

(2): this follows directly from (1), together with the fact that by the universal property of $\Spc \bK^c$, we have $\sigma(A)=\eta^{-1}(\suppB(A))$. 

(3): if $A \in \bK^c$ with $\sigma(A)=\varnothing,$ then (2) implies $\suppB(A)=\varnothing$, which by faithfulness of $\suppB$ implies that $A \cong 0$. 

(4): suppose $V$ is Thomason, that is, $V= \bigcup_{k \in K} V_k$ where $V_k$ is closed and $X \backslash V_k$ is quasicompact, for each $k$. Then $\rho^{-1}(V) = \bigcup_{k \in K} \rho^{-1}(V_k),$ so it suffices to show that each $\rho^{-1}(V_k)$ has quasicompact complement. Since $\Spc \bK^c \backslash \rho^{-1}(V_k) = \rho^{-1}(X \backslash V_k)$, it suffices to prove that the inverse image of a quasicompact open set is quasicompact.  

Suppose $U$ is a quasicompact open subset of $X$, and suppose $\rho^{-1}(U)$ has an open cover $\{U_i\}_{i \in I}$. We must show that it admits some finite subcover. We claim that the collection of open sets 
\[
\{ X \backslash (\rho(\Spc \bK^c \backslash U_i)) \}_{i \in I}
\]
forms an open cover of $U$ (each of these sets is open since $\rho$ is a closed map by (1)). Indeed, if $x \in U$, then $\eta(x) \in \rho^{-1}(U)$, so $\eta(x) \in U_i$ for some $i$; we claim that $x \in X \backslash (\rho(\Spc \bK^c \backslash U_i))$ for this particular $i$. Suppose to the contrary. Then $x \in \rho(\Spc \bK^c \backslash U_i)$, so there exists $\bP \not \in U_i$ with $x=\rho(\bP)$. But $\eta(x)= \eta \rho(\bP) \subseteq \bP$, so any open set containing $\eta(x)$ must contain $\bP$, and $\eta(x) \in U_i$, so $\bP \in U_i$, a contradiction. Our claim that $x \in X \backslash (\rho(\Spc \bK^c \backslash U_i)) $ follows. 

Now, since $U$ is quasicompact, there exists a finite subset $J \subseteq I$ such that
\[
\{ X \backslash (\rho(\Spc \bK^c \backslash U_j)) \}_{j \in J}
\]
covers $U$. We now claim that $\{U_j\}_{j \in J}$ covers $\rho^{-1}(U)$, at which point our proof will be complete. It is enough to check that each $\eta(x),$ for $x \in U$, is in some $U_j$ for $j \in J$, since again if $\bP \in \rho^{-1}(U)$ then $\bP$ is in any open subset containing $\eta \rho(\bP)$. Since $x \in U$, it is in some $X \backslash (\rho(\Spc \bK^c \backslash U_j))$ for some $j \in J$; we claim that $\eta(x) \in U_j$. Suppose to the contrary. Then $\eta(x) \in \Spc \bK^c \backslash U_j$, and so $x=\rho \eta(x) \in \rho(\Spc \bK^c \backslash U_j)$, a contradiction. Hence $\eta(x) \in U_j$, and the collection
\[
\{U_j\}_{j \in J}
\]
is indeed a finite cover of $\rho^{-1}(U)$, completing the proof.

(5): if $V \subseteq X$ is a closed subset, then it is Thomason, since $X$ is Noetherian. By (4), $\rho^{-1}(V)$ is Thomason in $\Spc \bK^c$, and so by \Cref{class-ideals} we have $\rho^{-1}(V) = \suppB(A)$ for some $A \in \bK^c$, using the assumption that $\bK^c$ is finitely-generated. But then $\sigma(A)=\rho(\supp^B(A)) = \rho(\rho^{-1}(V)) = V$, since $\rho$ is surjective. 
\end{proof}

We note the following property of comparative support data (this property will also appear in the following section).

\begin{proposition}\label{prop:comparative=existence_uniqueness_generic_points}
Let $(X, \sigma)$ be a comparative support on $\bK^c$. Then for each $\bP \in \Spc \bK^c,$ the closed set $\eta^{-1}(\mc{V}(\bP)) = \eta^{-1}(\overline{\{\bP\}})$ has a unique generic point.
\end{proposition}
\begin{proof}

If $\rho$ is the comparison map for $\sigma$, we claim that $\rho(\bP)$ is the unique generic point of $\eta^{-1}(\mc{V}(\bP))$, i.e. that $\mc{V}(\rho(\bP))=\eta^{-1}(\mc{V}(\bP))$. First, by the comparative hypothesis, $\eta(\rho(\bP))\in \mc{V}(\bP)$, so $\rho(\bP)\in \eta^{-1}(\mc{V}(\bP))$. But this is a closed set, so $\mc{V}(\rho(\bP))\subseteq \eta^{-1}(\mc{V}(\bP))$. It remains only to prove the reverse inclusion. Take any $x\in\eta^{-1}(\mc{V}(\bP))$, hence $\eta(x)\in\mc{V}(\bP)$. Since $\rho$ is closed by Lemma \ref{comparative-props}(1), 
\[ x=\rho\circ \eta(x)\in\rho(\mc{V}(\bP))=\mc{V}(\rho(\bP)),\]
and so $\eta^{-1}(\mc{V}(\bP))\subseteq \mc{V}(\rho(\bP)).$
\end{proof}
\bre{unique-compar}
In fact, this proof implies that if $\sigma$ is comparative, then the comparison map $\rho$ is unique, since it can be constructed from $\eta:X\hookrightarrow \Spc\bK^c$ by setting $\rho(\bP)$ as the unique generic point of $\eta^{-1}(\mc{V}(\bP))$. In fact, for any tensorial support $(X,\sigma)$ such that the universal map $\eta$ satisfies the condition that $\eta^{-1}(\mc{V}(\bP))$ has a unique generic point for every $\bP\in\Spc\bK^c$, we can always construct a (potentially non-continous, as the next example shows) map $\widetilde{\rho}:\Spc\bK^c\to X$ by sending $\bP$ to the unique generic point of $\eta^{-1}(\mc{V}(\bP))$. To see that this map can fail to be continuous, even for Noetherian spectral spaces, consider the Noetherian spectral space characterized by the following specialization poset (see Notation \ref{spec-poset} and Remark \ref{every-spec-poset}):
\[\Spc \bK^c=\begin{tikzcd}
	&& x \\
	\dotsm & {y_{-1}} & {y_0} & {y_1} & \dotsm \\
	&& z
	\arrow[from=1-3, to=2-1]
	\arrow[from=1-3, to=2-2]
	\arrow[from=1-3, to=2-3]
	\arrow[from=1-3, to=2-4]
	\arrow[from=1-3, to=2-5]
	\arrow[from=2-1, to=3-3]
	\arrow[from=2-2, to=3-3]
	\arrow[from=2-3, to=3-3]
	\arrow[from=2-4, to=3-3]
	\arrow[from=2-5, to=3-3]
\end{tikzcd}\]
In this topological space, we have a unique generic point $x$, a unique closed point $z$, and infinitely many points $y_i$, indexed by $\mathbb{Z}$, where the closure of $\{y_i\}$ is the set $\{y_i,z\}$ for each $i$. It is not hard to check that the topology on this space implies that the closed sets are the entire space, the empty set, $\{z\}$, and the union of $\{z\}$ union with finitely many of $y_i$'s. On the other hand, define another space $X$ given by two points: $\{x',z'\}$ with $x'$ the generic point and $z'$ the closed point. Now we define the map $\eta$ as follows:
\[\begin{tikzcd}
	X &&&& {\Spc \bK^c} \\
	{x'} &&&& x \\
	&& \dotsm & {y_{-1}} & {y_0} & {y_1} & \dotsm \\
	{z'} &&&& z
	\arrow["\eta",  dotted, from=2-1, to=2-5]
	\arrow[from=2-1, to=4-1]
	\arrow[from=2-5, to=3-3]
	\arrow[from=2-5, to=3-4]
	\arrow[from=2-5, to=3-5]
	\arrow[from=2-5, to=3-6]
	\arrow[from=2-5, to=3-7]
	\arrow[from=3-3, to=4-5]
	\arrow[from=3-4, to=4-5]
	\arrow[from=3-5, to=4-5]
	\arrow[from=3-6, to=4-5]
	\arrow[from=3-7, to=4-5]
	\arrow["\eta", dotted, from=4-1, to=4-5]
\end{tikzcd}\]
Hence $\eta$ sends $x'\mapsto x$ and $z'\mapsto z$. We can then define the support datum $(X,\sigma)$ by defining $\sigma(A)\coloneqq \eta^{-1}(\suppB(A))$ for each $A\in\bK^c$. It is not hard to check that this datum is tensorial, and $\eta^{-1}(\mc{V}(\bP))$ has a unique generic point for all $\bP \in \Spc \bK^c$; furthermore, both $X$ and $\Spc \bK^c$ are Noetherian spectral spaces. Attempting to define $\widetilde{\rho}$ would send $z\mapsto z'$, $x\mapsto x'$, and all $y_i\mapsto z'$. But this is not continuous, as $\{z'\}$ is certainly closed but its preimage $\{z,y_i\}_{i\in \mathbb{Z}}$ is not.
\ere
The next lemma is the primary tool which allows us to transport arguments using the bijection between Thomason subsets of $\Spc \bK^c$ and thick ideals of $\bK^c$ to arguments based on the support $(X,\sigma).$ 

\begin{lemma}
\label{theta-cup}
Let $\sigma$ be a comparative support with value in a Zariski space $X$ on a small monoidal triangulated category $\bK^c$ satisfying (\ref{kc-condition}). Let $V$ and $W$ be specialization-closed subsets of $X$. We have $\langle \Theta(V) \cup \Theta(W) \rangle = \Theta(V \cup W)$.    
\end{lemma}

Recall that $\Theta(S)$ is the thick ideal $\{A \in \bK^c : \sigma(A) \subseteq S\}$ for any specialization-closed subset $S$ of $X$. 

\begin{proof}
    It is clear by definition that $\langle \Theta(V) \cup \Theta(W) \rangle \subseteq \Theta(V \cup W)$. To show the other direction, suppose $\sigma(A) \subseteq V \cup W$. Since $\sigma(A)= \rho(\suppB(A))$ by Lemma \ref{comparative-props}(2), $\suppB(A) \subseteq \rho^{-1}(V) \cup \rho^{-1}(W)$. Since $V$ and $W$ are Thomason (indeed, every specialization-closed subset of $X$ is Thomason, since $X$ is Noetherian), we know that $\rho^{-1}(V)$ and $\rho^{-1}(W)$ are Thomason, by Lemma \ref{comparative-props}(4). We also know that $\suppB(A)$ is Thomason, by \Cref{class-ideals}. It follows that $V':=\suppB(A) \cap \rho^{-1}(V)$ and $W' := \suppB(A) \cap \rho^{-1}(W)$ are Thomason, and by \Cref{class-ideals}, there are thick ideals $\bI$ and $\bJ$ with 
    \[
    \Phi_{\suppB}(\bI)=V', \; \; \; \Phi_{\suppB}(\bJ) = W'.
    \]
    Again by \Cref{class-ideals}, since $\suppB(A) \subseteq V' \cup W'$, we have $A \in \langle \bI, \bJ \rangle$. But now note that since $\suppB(B) \subseteq V' \subseteq \rho^{-1}(V)$ for all $B \in \bI$, and $\suppB(C) \subseteq W' \subseteq \rho^{-1}(W)$ for all $C \in \bJ$, we have that $\bI \subseteq \Theta(V)$ and $\bJ \subseteq \Theta(W)$. That is, we have
    \[
    A \in \langle \bI, \bJ \rangle \subseteq \langle \Theta(V), \Theta(W) \rangle,
    \]
    and the proof is complete.

\end{proof}

We now leverage Lemma \ref{theta-cup} to prove several useful identities for Rickard idempotents associated to specialization-closed subsets of $X$.

\begin{corollary}
\label{gamma-union}
Let $\sigma$ be a comparative support with value in a Zariski space $X$ on a small monoidal triangulated category $\bK^c$ satisfying (\ref{kc-condition}). Let $V$ and $W$ be specialization-closed subsets of $X$. Then $L_{V \cup W} \unit \cong L_V L_W \unit.$
\end{corollary}

\begin{proof}
Applying $L_V L_W$ to the distinguished triangle
    \[
    \itGamm_{V \cup W} \unit  \to \unit \to L_{V \cup W} \unit, 
    \]
    we get
    \[
    L_V L_W \itGamm_{V \cup W} \unit  \to L_V L_W\unit \to L_{V \cup W} \unit, 
    \]
    by Lemma \ref{lemma:commutes_with_Gamma_L}. But
 $L_V L_W A \in \Loc( \Theta(V) \cup \Theta(W) )^{\perp}$ for any object $A$, by \cite[Lemma 3.2.2(a)]{NVY4}. Now note that $L_V L_W \itGamm_{V \cup W} \unit \in \Loc(\Theta(V \cup W)) = \Loc(\langle \Theta(V) \cup \Theta(W) \rangle)$ by Lemma \ref{theta-cup}, and it follows that $L_V L_W \itGamm_{V \cup W} \unit \cong 0$, hence $L_V L_W \unit \cong L_{V \cup W} \unit$. 
\end{proof}

\begin{lemma}\label{cor:L_Gamma_z}
Let $\sigma$ be a comparative support with value in a Zariski space $X$ on a small monoidal triangulated category $\bK^c$ satisfying (\ref{kc-condition}). Suppose $Y,Z$ are specialization-closed subsets of $X$. Then
$L_Z \unit \otimes \itGamm_Y \unit $ depends only on $Y\cap Z^c$.
\end{lemma}
\begin{proof}
The proof is identical to \cite[Lemma~7.4]{BF2011}; the main technical difficulty in adapting their proof, which is for the Balmer support, to an arbitrary comparative support is Lemma \ref{theta-cup} above. We include the argument for the convenience of the reader. 
It suffices to show that $L_Z \itGamm_Y \unit \cong L_Z\itGamm_{Y'} \unit$ for $Y\cap Z^c=Y'\cap Z^c$ and $L_Z\itGamm_Y \unit \cong L_{Z'}\itGamm_Y \unit$ for $Y\cap Z^c=Y\cap {Z'}^c$, since in this case if $Y \cap Z^c = Y' \cap Z'^c$ we can check
\[
L_Z \itGamm_Y \unit \cong L_Z \itGamm_{Y \cap Y'} \unit \cong L_{Z \cup Z'} \itGamm_{Y \cap Y'} \unit \cong L_{Z'} \itGamm_{Y'} \unit.
\]
The two statements are proved the same way, so we will just show the first. In fact, it is enough to show that $L_Z \itGamm_Y \unit \cong L_Z \itGamm_{Y'} \unit$ for $Y \subseteq Y'$. In this case, we have a map of exact triangles
\[\begin{tikzcd}
	{\itGamm_Y \unit}  & \unit & {L_Y \unit} & {{}} \\
	{\itGamm_{Y’} \unit} & \unit & {L_{Y’}\unit} & {{}}
	\arrow[from=1-1, to=1-2]
	\arrow[from=1-1, to=2-1]
	\arrow[from=1-2, to=1-3]
	\arrow[Rightarrow, no head, from=1-2, to=2-2]
	\arrow[from=1-3, to=1-4]
	\arrow[from=1-3, to=2-3]
	\arrow[from=2-1, to=2-2]
	\arrow[from=2-2, to=2-3]
	\arrow[from=2-3, to=2-4]
\end{tikzcd}\]
by Lemma \ref{lemma:commutes_with_Gamma_L}(6). Note that $Y\cup Z=(Y\cap Z^c)\cup Z=(Y'\cap {Z}^c)\cup Z=Y'\cup Z$, so
Corollary \ref{gamma-union} implies that we have an isomorphism
\[ L_Y \unit \otimes L_Z \unit \cong  L_{Y\cup Z} \unit \cong  L_{Y'\cup Z} \unit \cong L_{Y'} \unit \otimes L_Z \unit.\]
This means that when we apply the exact functor $-\otimes L_Z \unit $ to the triangles above, the second and third vertical maps become isomorphisms, hence $\itGamm_Y \unit \otimes L_Z \unit \cong \itGamm_{Y'}\unit \otimes L_Z \unit$.
\end{proof}

We next prove some preliminary results regarding the localizing tensor ideals generated by Rickard idempotents. Recall that the localizing tensor ideal generated by a collection of objects $\bS$ is denoted $\Loc^{\otimes}(\bS)$. 

\ble{gamma-union}
Let $\sigma$ be a comparative support with value in a Zariski space $X$ on a small monoidal triangulated category $\bK^c$ satisfying (\ref{kc-condition}). Let $V=\bigcup_{i \in I} V_i$. We have $\Loc^{\otimes} ( \itGamm_{V_i} \unit : i \in I) = \Loc^{\otimes} (\itGamm_V \unit)$.
\ele

\begin{proof}
On one hand, since $\itGamm_{V_i} \unit \otimes \itGamm_V \unit \cong \itGamm_{V_i} \unit$, it is clear that 
\[
\Loc^{\otimes} ( \itGamm_{V_i} \unit : i \in I) \subseteq \Loc^{\otimes} (\itGamm_V \unit).
\]

On the other hand, since $\itGamm_V\unit$ is by definition in the localizing subcategory generated by $\Theta(V)$, it is enough to show that $\Theta(V)$ is contained in $\Loc^{\otimes}(\itGamm_{V_i} \unit : i \in I).$ Suppose that $A \in \Theta(V)$. Since $\sigma(A)$ is closed and $X$ is a Zariski space, we can write $\sigma(A)$ as a finite union of irreducible closed sets, each of which has a unique generic point. In particular, there is some finite collection $V_{i_1},..., V_{i_n}$ such that 
\[
\sigma(A) \subseteq V_{i_1} \cup... \cup V_{i_n}.
\]
Hence $A \in \Theta(V_{i_1} \cup... \cup V_{i_n}) = \langle \Theta(V_{i_1}),..., \Theta(V_{i_n}) \rangle$, using Lemma \ref{theta-cup}. But now note that
\[
\Loc(\langle \Theta(V_{i_1}),..., \Theta(V_{i_n}) \rangle) \subseteq \Loc^{\otimes} (\itGamm_{V_{i_1}} \unit ,..., \itGamm_{V_{i_n}} \unit),
\]
since if $B \in \Theta(V_{i_j})$ then $B \otimes \itGamm_{V_{i_j}} \unit \cong B$. The lemma now follows. 
\end{proof}

\begin{proposition}\label{lemma:4.17}
Let $\sigma$ be a comparative support with value in a Zariski space $X$ on a small monoidal triangulated category $\bK^c$ satisfying (\ref{kc-condition}). Let $Y\subseteq X$ be some subset. Suppose $A\in\bK$ satisfies $\itGamm_x A=0$ for all $x\not\in Y$. Then $A\in \Loc^{\otimes}(\itGamm_y \unit : y\in Y)$. 
\end{proposition}
\begin{proof}
The proof closely follows \cite[Theorem~3.22]{BHS2023b}. Set 
\[
\bL \coloneqq \Loc^{\otimes}(\itGamm_y \unit : y\in Y).
\]
Define $Z\subseteq X$ to be
\[ Z\coloneqq \{x\in X: \itGamm_{\mc{V}(x)}A\in \bL\}.\]
This is a specialization closed set. We claim that $Z=X$; this would prove the lemma, since then $\itGamm_Z \unit \cong \itGamm_X \unit \cong\unit$, so that by \leref{gamma-union}
\[A\cong  \itGamm_Z A\in \Loc^{\otimes}(\itGamm_{\mc{V}(x)}A: x\in Z)\subseteq \bL.\]

To prove the claim, suppose for the sake of contradiction that there exists $x\in X\setminus Z$; choose $x$ minimal with respect to specialization (which exists because $X$ is Noetherian). Then $S\coloneqq \mc{V}(x)\setminus \{x\}\subseteq Z$ is specialization-closed (as is $\mc{V}(x)$). Therefore we have
\[\itGamm_S\unit\otimes \itGamm_{\mc{V}(x)}\unit\cong \itGamm_S\unit,\]
and Lemma~\ref{cor:L_Gamma_z} implies that
\[ \itGamm_x \unit = \itGamm_{\mc{V}(x)}\unit\otimes L_S\unit.\]
Applying the exact functor $\itGamm_{\mc{V}(x)}$ to the distinguished triangle
\[\itGamm_S A\to A \to L_S A\to, \]
we obtain
\[ \itGamm_S A\to \itGamm_{\mc{V}(x)}A\to \itGamm_x A\to .\]
We claim that the first and third terms are in $\bL$. 

On one hand, by \leref{gamma-union} we have 
\[
\Loc^{\otimes}(\itGamm_S \unit)=\Loc^{\otimes}(\itGamm_{\mc{V}(z)} \unit : z \in S),
\]
which implies that $\itGamm_{S} A$ is in the localizing ideal generated by $\itGamm_{\mc{V}(z)} A$ for all $z \in S$. Since $S \subseteq Z,$ each $\itGamm_{\mc{V}(z)} A$ is in $\bL$, so $\itGamm_S A \in \bL$ as well. On the other hand, $\itGamm_x A\in \bL$ because if $x \in Y$ then $\itGamm_x \unit \in \bL$ by definition, and if $x \not \in Y$ then $\itGamm_x A \cong 0$, which is also in $\bL$. 

Therefore $\itGamm_{\mc{V}(x)}A\in\bL$, implying that $x\in Z$, contradicting the assumption that $x\not\in Z$. Thus $Z\cong X$ and the claim is proved.
\end{proof}

\begin{theorem}
\label{comp-faith}
Let $\sigma$ be a comparative support with value in a Zariski space $X$ on a small monoidal triangulated category $\bK^c$ satisfying (\ref{kc-condition}). Then the extension $\widetilde{\sigma}$ of $\sigma$ is faithful.
\end{theorem}

\begin{proof}
The argument will be similar to the one given by Stevenson for the Balmer support in the commutative case \cite[Theorem 4.19(2)]{Stevenson2013}, although without the assumption that $\bK$ arises as the homotopy category of a monoidal model category (this assumption was used in the analogue of Proposition~\ref{lemma:4.17}, but was removed in \cite[Theorem 3.22, Remark 3.24]{BHS2023b}).

First, note that  Proposition~\ref{lemma:4.17} implies that
\[ \bK=\Loc^{\otimes}(\itGamm_x \unit : x\in X),\] since $\unit$ is in $\Loc^{\otimes}(\itGamm_x \unit : x\in X).$
From this, it follows that
\[
\Loc^{\otimes}(A) = \Loc^{\otimes}(\itGamm_x A : x \in X).
\]
Therefore 
\[ A\cong 0\iff \Loc^{\otimes}(A)=\{0\} \iff \itGamm_x A\cong0\;\;\forall\; \; x\in X\iff \widetilde{\sigma}(A)=\varnothing.\]
\end{proof}

In particular, since the Balmer support is always comparative, we may now deduce the following. 

\bco{balmer-faith}
Let $\bK^c$ be a small $\Spc$-Noetherian monoidal triangulated category. Then the extended Balmer support $\SuppB$ on $\bK$ is faithful.
\eco

Lastly, we note that using the work of this section we can prove that the extension of the Balmer support we have given is the unique faithful tensorial extension of $\suppB$.

\begin{proposition}\label{uniqueness_of_balmer_extension}
Let $\bK^c$ be a small $\Spc$-Noetherian monoidal triangulated category. Suppose that $\widehat{\sigma}$ is a faithful and tensorial extension of the Balmer support $\suppB$ from $\bK^c$ to $\bK$. Then $\widehat{\sigma}=\SuppB$.
\end{proposition}

\begin{proof}
We will use throughout the proof that $\SuppB$ is both tensorial and faithful (\prref{tensor-balmer} and \coref{balmer-faith}). First, we claim that it suffices to show that $\widehat{\sigma}(\itGamm_{\bP}\unit)=\{ \bP\}$ for all $\bP \in \Spc \bK^c$, by the following argument. Assume the above statement; we must show that for arbitrary $A\in \bK$, we have $\bP \in \widehat{\sigma}(A)$ if and only if $\bP \in \SuppB(A)$. For the first direction, suppose that $\bP \in \SuppB(A)$. Note that
\[
\bP\in \SuppB(A)\iff \itGamm_{\bP} \unit\otimes A\not \cong 0\iff \widehat{\sigma}(\itGamm_{\bP}\unit\otimes A)\ne\varnothing,
\]
using the assumed faithfulness of $\widehat{\sigma}$. 
By tensoriality, $\widehat{\sigma}(\itGamm_{\bP}\unit\otimes A)\subseteq \widehat{\sigma}(\itGamm_{\bP}\unit)=\{\bP\}$, but since it is not empty, it must be equal to $\{\bP\}$. Again by tensoriality, $\{\bP\}=\widehat{\sigma}(\itGamm_{\bP}\unit\otimes A)\subseteq \widehat{\sigma}(A)$, so we find that $\bP \in\SuppB(A)$ implies that $\bP\in \widehat{\sigma}(A)$. 
On the other hand suppose $\bP\in \widehat{\sigma}(A)$. Now since (see the proof of Theorem \ref{comp-faith})
\[ \Loc^\otimes(A)=\Loc^\otimes(\itGamm_{\bQ} \unit\otimes A: \bQ\in \Spc \bK^c), \]
by \leref{support-local}, it must be the case that $\bP\in \widehat{\sigma}(\itGamm_{\bQ}\unit \otimes A)$ for some $\bQ\in \Spc \bK^c$. But by our assumptions, $\{\bQ\}=\widehat{\sigma}(\itGamm_{\bQ}\unit)\supseteq \widehat{\sigma}(\itGamm_{\bQ}\otimes A)$. Therefore $\bP=\bQ$ and $\bP\in \widehat{\sigma}(\itGamm_{\bP}\unit \otimes A)$; in particular $\itGamm_{\bP}\unit \otimes A\not \cong 0$ and hence $\bP\in\SuppB(A)$. This shows our first claim.

To complete the proof, we now just need to show that $\widehat{\sigma}(\itGamm_{\bP})=\{\bP\}$ for each $\bP \in \Spc \bK^c$. Recall that since $\mc{V}(\bP)$ is closed, we have some $A\in \bK^c$ with $\suppB(A)=\mc{V}(\bP)$. Note that $\itGamm_{\bP}\unit\in \Loc^\otimes(\itGamm_{\mc{V}(\bP)}\unit)$ by the definition of $\itGamm_{\bP} \unit$ (see \notaref{confusing-notation-two}). On the other hand $\itGamm_{\mc{V}(\bP)}\unit\in \Loc^\otimes(\Theta(\mathcal{V}(\bP)))=\Loc^\otimes(A)$ by the definition of $\itGamm_{\mc{V}(\bP)}$ and by Theorem~\ref{class-ideals}. All this is to say, $\itGamm_{\bP}\unit \in \Loc^\otimes(A)$. By \leref{support-local}, any object $B\in \Loc^\otimes(A)$ satisfies
$\widehat{\sigma}(B)\subseteq \widehat{\sigma}(A)=\suppB(A)=\mc{V}(\bP)$. It follows that 
\[ \widehat{\sigma}(\itGamm_{\bP} \unit)\subseteq \mc{V}(\bP).\]

Now take any $\bQ\in \mc{V}(\bP)$ for $\bQ\ne \bP$. Then as before, by realization for the Balmer support there exists $B\in \bK^c$ with $\suppB(B)=\mc{V}(\bQ)$. Since $\bP \not \in \suppB(B)=\SuppB(B)$, we have $\itGamm_{\bP} \unit \otimes B\cong 0$. Therefore, by tensoriality of $\widehat{\sigma}$ (using that $B\in\bK^c$) we have 
\[ \varnothing=\widehat{\sigma}(\itGamm_{\bP} \unit \otimes B)=\widehat{\sigma}(\itGamm_{\bP} \unit)\cap \widehat{\sigma}(B)=\widehat{\sigma}(\itGamm_{\bP}\unit)\cap \mc{V}(\bQ).\]
In particular for all $\bQ\in \mc{V}(\bP)$ with $\bQ\ne \bP$, we have $\bQ\not\in \widehat{\sigma}(\itGamm_{\bP}\unit)$. This implies $\widehat{\sigma}(\itGamm_{\bP}\unit)$ is either $\varnothing$ or $\{\bP\}$. But $\itGamm_{\bP}\unit\not \cong 0$, we conclude that $\widehat{\sigma}(\itGamm_{\bP}\unit)=\{\bP\}$, completing the proof.
\end{proof}
\bre{replace-balmer-uniqueness}
Note that the proof of Proposition \ref{uniqueness_of_balmer_extension} uses critically that the Balmer support is tensorial, faithful, realizing, and based on a Zariski space, and has a tensorial and faithful extension. Recall that any other support satisfying these conditions is homeomorphic to the Balmer support by Theorem~\ref{nvy-restated}. Hence this present proof does not immediately extend to other support data of the form $(X, \sigma)$ where $X$ is not homeomorphic to $\Spc \bK^c$.
\ere

\section{Faithfulness of extended realizing tensorial support}
\label{sect-tensorial}

Throughout this section, assume that $(X,\sigma)$ is a realizing tensorial support, and drop the assumption from the previous section that $\sigma$ is comparative. We will assume that $\Spc \bK^c$ is Noetherian. In this section we determine when the extension $\widetilde{\sigma}$ is faithful, giving a characterization in terms of the universal map $\eta: X \to \Spc \bK^c$.

Recall that by \prref{eta-inj}, since $\sigma$ is realizing, we have that $\eta$ is injective, a fact that we use throughout the section. We note the following characterization of faithfulness of $\sigma$ in our setting in terms of the image of $\eta$; this lemma will not be used explicitly for the results of this section, but is helpful to keep in mind for the purpose of constructing examples.
\begin{lemma}
\label{lemma:all_closed_points_in_X}
Let $(X,\sigma)$ be a realizing tensorial support on a small $\Spc$-Noetherian monoidal triangulated category $\bK^c$. Then $\sigma$ is faithful if and only if the image of $\eta$ contains all closed points of $\Spc\bK^c$.
\end{lemma}
\begin{proof}
Assume $\sigma$ faithful. If $\bP$ is a closed point of $\Spc \bK^c$, we can find an object $A$ such that $\suppB(A)=\{\bP\}$ by \Cref{class-ideals} and the Noetherianity of $\Spc \bK^c$. Since $\sigma(A)=\eta^{-1}(\suppB(A))$, and since $\sigma$ is faithful, there must be a point in the preimage of $\bP$.

For the other direction, if all closed points are in the image of $\eta$, and if $A$ is in $\bK^c$, then $\sigma(A)=\eta^{-1}(\suppB(A))$ is nonempty since $\suppB$ is faithful and closed sets of Noetherian topological spaces always contain closed points.
\end{proof}

We give examples below where, even though $\sigma$ is faithful, the extended support $\widetilde{\sigma}$ will fail to be faithful. Our goal this section will be to describe precisely when $\widetilde{\sigma}$ is faithful. In order to do that, we will introduce a convenient format for expressing the data of the underlying topological spaces of our supports.

\bnota{spec-poset}\label{spec-poset}
    Let $X$ be a topological space. We will use the convention that the specialization poset of $X$ will be defined by $x \leq y$ if and only if $x \in \mc{V}(y)=\overline{\{y\}}$, for all $x$ and $y$ in $X$. We write the directed graph of the specialization poset by drawing an edge $x \to y$ if and only if $y\in\mc{V}(x)$ and there does not exist another point $z\in X$ such that $\mc{V}(x)\supsetneq \mc{V}(z)\supsetneq \mc{V}(y)$ (i.e., $y$ is a specialization of $x$ in the minimal sense).
\enota

\bre{every-spec-poset}\label{every-spec-poset}
Recall that a space is called {\it{spectral}} if it is isomorphic to the spectrum of a commutative ring. Noetherian spectral spaces are completely characterized by their specialization posets: two Noetherian spectral spaces are homeomorphic if and only if their specialization posets are isomorphic \cite[Proposition~1.26]{tedd2017ring} (although if we no longer require Noetherianity, a specialization poset can be given distinct topologies which still give it the structure of a spectral space). Every finite poset arises as the specialization poset of a finite spectral space (see \cite[\S3]{tedd2017ring}, which also surveys several methods to construct said rings). For any ring $R$ with Noetherian spectrum, we have a homeomorphism $\Spec R\cong \Spc\bK^c$ where $\bK=\der(R)$ and $\bK^c=\dperf(R)$ \cite{Thomason1997} where $\der(R)$ is the derived category of $R$-modules and $\dperf(R)$ is the perfect derived category of $R$-modules. Therefore any finite poset is the specialization poset of the Balmer spectrum of some small monoidal triangulated category, which we use to construct examples below.
\ere

We begin by giving an example of a tensorial realizing support whose extension is not faithful. Note that in the directed graphs of the specialization posets of $X$ and $\Spc \bK^c$, closed points are always the end nodes and every connected component of the specialization poset contains at least one closed point.

\bex{z-one-over-p}
Consider the following map of specialization posets:
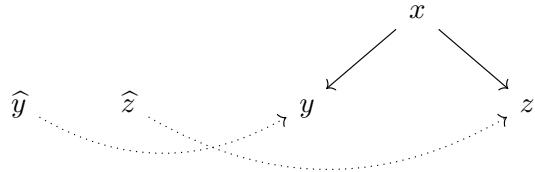
\[\begin{tikzcd}
                                       &                                         &  &   & x \arrow[ld] \arrow[rd] &   \\
\widehat y \arrow[rrr, dotted, bend right] & \widehat z \arrow[rrrr, dotted, bend right] &  & y &                         & z
\end{tikzcd}\]
We can choose $\bK^c$ such that the topological space associated to the right side, given by three points $\{x,y,z\}$ such that $y,z$ are closed points and $\overline{\{x\}}=\{x,y,z\}$, is homeomorphic to $\Spc\bK^c$ by the discussion above. This spectral space is concretely realized as $\textnormal{Spec}\, R\cong \Spc \dperf(R)\cong\Spc \der (R)^c$, where $R=\mathbb{Z}[\frac{1}{p}: p\ne 2,3]$. Now we consider an alternate support based on the topological space $X=\{\widehat{y},\widehat{z}\}$ consisting of two distinct closed points. We have the continuous map $\eta: X \to \Spc \bK^c$ sending $\widehat{y}$ to $y$ and $\widehat{z}$ to $z$, as in the dotted arrows in the diagram above. This induces a support datum $\sigma$ based on $X$, where we define $\sigma(A)=\eta^{-1}(\suppB(A))$. This support is tensorial and realizing, which we can observe directly from the corresponding properties of $\suppB$. Consider the object
\[A\coloneqq \itGamm_{x} \unit \cong \itGamm_{\mathcal{V}(x)}\unit \otimes L_{\mathcal{Z}(x)} \unit \cong \itGamm_{\{x,y,z\}} \unit \otimes L_{\{y,z\}} \unit \cong L_{\{y,z\}} \unit\in \bK.\]
This object has the property that $\SuppB(A)=\{x\}$, hence it is nonzero. We will see that $\widetilde{\sigma}(A)=\varnothing$. Let us examine the elements $\itGamm_{\widehat{y}} A$ and $\itGamm_{\widehat{z}} A$. We have
\begin{align*}
    \itGamm_{\widehat{y}} \unit &\cong  \itGamm_{\{\widehat{y}\}} \unit \otimes L_{\{\widehat{z}\}}\unit ,\\
    \itGamm_{\widehat{z}} \unit  &\cong \itGamm_{\{\widehat{z}\}} \unit \otimes L_{\{\widehat{y}\}}\unit.
\end{align*}
Computing the ideals directly, we see by the definition of the support $\sigma$ that 
\begin{align*}
    \Theta(\{\widehat{z}\}) &= \Theta(\{z\}),\\
    \Theta(\{\widehat{y}\}) &= \Theta(\{y\}).
\end{align*}
This implies that 
\begin{align*}
    \itGamm_{\widehat{y}} \unit \otimes A &\cong \itGamm_{\{\widehat{y}\}} \unit \otimes L_{\{\widehat{z}\}} \unit \otimes L_{\{y,z\}} \unit \cong \itGamm_{\{y\}}\unit\otimes L_{\{y,z\}}\unit\cong 0,\\
    \itGamm_{\widehat{z}} \unit \otimes A &\cong \itGamm_{\{\widehat{z}\}} \unit \otimes L_{\{\widehat{y}\}} \unit \otimes L_{\{y,z\}} \unit \cong \itGamm_{\{z\}}\unit\otimes L_{\{y,z\}}\unit\cong 0,\\
    &\implies \widetilde{\sigma}(A)=\varnothing.
\end{align*}
This shows that we have a nontrivial support $(X,\sigma)$ which is faithful, exhibits realization and tensor product property, and yet its extension $\widetilde{\sigma}$ is not faithful.
\eex

The main result of this section characterizes exactly when the extension of a tensorial support is faithful; we deduce it from the following theorem, which tells us when the extended support of the $\itGamm_{\bP} \unit$ objects are empty.

\begin{theorem}\label{thm:O_2_faithful}
Let $(X,\sigma)$ be a realizing faithful tensorial support on a small $\Spc$-Noetherian monoidal triangulated category $\bK^c$ with extension $\widetilde{\sigma}$, and let $\eta:X\to\Spc\bK^c$ be the universal map. Then for $\bP \in \Spc \bK^c$, we have $\widetilde{\sigma}(\itGamm_{\bP}\unit )=\varnothing$ if and only if $\eta^{-1}(\mc{V}(\bP))$ does not have a unique generic point.
\end{theorem}

Recall that, in terms of the specialization poset, a generic point is a unique vertex which is the root of all points.

\begin{proof}
Since $\eta$ is injective, if $\bP \in \Spc \bK^c$ is in the image of $\eta$, it has a unique preimage; for ease of notation, denote this preimage by $\widehat{\bP}$, if it exists.

First, $\widetilde{\sigma}(\itGamm_{\bP} \unit )\ne\varnothing$ if and only if there exists $\bQ\in\Spc\bK^c$ with preimage $\widehat{\bQ}\in X$ such that $\itGamm_{\widehat{\bQ}} \unit \otimes \itGamm_{\bP}\unit \not \cong 0$. Since $\itGamm_{\widehat{\bQ}} \unit \cong \itGamm_{\mc{V}(\widehat{\bQ})} \unit \otimes L_{\mathcal{Z}(\widehat{\bQ})} \unit$, by the tensor product property from \leref{tensor-idem-intersect}, this is true if and only if $\bP\in\SuppB(\itGamm_{\mc{V}(\widehat{\bQ})} \unit)$ and $\bP\in\SuppB(L_{\mathcal{Z}(\{\widehat{\bQ})\}} \unit )$. We will see that these two conditions force $\widehat{\bQ}$ to be the unique generic point of $\eta^{-1}(\mc{V}(\bP))$.

By \leref{rickard-idem-supp},  
\begin{align*}
    \SuppB(\itGamm_{\mc{V}(\widehat{\bQ})} \unit) &= \Phi_{\suppB}(\Theta(\mc{V}(\widehat{\bQ})))\\
    &=\{\bR \in \Spc \bK^c: \exists A \in \Theta(\mc{V}(\widehat{\bQ})) \text{ with } \bR \in \suppB(A)\}\\
    &=\{\bR \in \Spc \bK^c : \exists A \text{ with } \sigma(A) \subseteq \mc{V}(\widehat{\bQ}) \text{ and } \bR \in \suppB(A)\}\\
    &=\{\bR \in \Spc \bK^c : \exists A \text{ with } \eta^{-1}(\suppB(A)) \subseteq \mc{V}(\widehat{\bQ}) \text{ and } \bR \in \suppB(A)\}\\
    &= \{\bR \in \Spc \bK^c : \eta^{-1} ( \mc{V}(\bR)) \subseteq  \mc{V}(\widehat{\bQ})\}.
\end{align*}
Hence $\bP\in\SuppB(\itGamm_{\mc{V}(\widehat{\bQ)}} \unit)$ if and only if $\eta^{-1}(\mc{V}(\bP)) \subseteq \mc{V}(\widehat{\bQ}).$ An analogous computation shows that $\bP\in\SuppB(L_{\mathcal{Z}(\widehat{\bQ})} \unit )$ if and only if $\eta^{-1}(\mc{V}(\bP)) \not \subseteq \mathcal{Z}(\widehat{\bQ}).$ But all points of $\mc{V}(\widehat{\bQ})$ except for $\widehat{\bQ}$ are also in $\mc{Z}(\widehat{\bQ})$ (recall that we are assuming $X$ is $T_0$). To summarize, then, we have that $\widehat{\bQ} \in \widetilde{\sigma}(\itGamm_{\bP} \unit)$ if and only if  
\[
\widehat{\bQ} \in \eta^{-1}(\mc{V}(\bP)) \subseteq \mc{V}(\widehat{\bQ}).
\]
But then we have
\[
\eta^{-1}(\mc{V}(\bP)) = \mc{V}(\widehat{\bQ}),
\]
that is, $\widehat{\bQ}$ is a generic point of $\eta^{-1}(\mc{V}(\bP)).$

\end{proof}
\begin{theorem}
\label{tensor-faith-iff}
Let $(X,\sigma)$ be a realizing faithful tensorial support on a small $\Spc$-Noetherian monoidal triangulated category $\bK^c$ with extension $\widetilde{\sigma}$. Then $\widetilde{\sigma}$ is faithful if and only $\eta^{-1}(\mc{V}(\bP))$ has a unique generic point for all $\bP \in \Spc \bK^c$.
\end{theorem}
\begin{proof}
By Theorem \ref{thm:O_2_faithful}, if $\widetilde{\sigma}$ is faithful then for each $\bP \in \Spc \bK^c$, the closed subset $\eta^{-1}(\mc{V}(\bP))$ has a unique generic point. 

For the other direction, suppose that every $\bP \in \Spc \bK^c$ satisfies the condition that $\eta^{-1}(\mc{V}(\bP))$ has a unique generic point, and let $A \in \bK$ be nonzero. We must show that $\widetilde{\sigma}(A) \not = \varnothing.$ Since we know that $\SuppB$ is faithful, by \coref{balmer-faith}, we know that there exists $\bP \in \Spc \bK^c$ such that $\itGamm_{\bP} \unit \otimes A \not \cong 0$. Let $\bQ \in \Spc \bK^c$ be the point in the image of $\eta$ such that $\widehat{\bQ} \in X$ is the unique generic point of $\eta^{-1}(\mc{V}(\bP))$, in particular $\bQ$ is in the closure of $\{\bP\}$. We claim that $\itGamm_{\widehat{\bQ}} \unit \otimes A \not \cong 0$. Indeed, it is enough to prove that $\itGamm_{\widehat{\bQ}} \unit \otimes \itGamm_{\bP} \unit \cong \itGamm_{\bP}\unit$. To prove this, it is now enough to prove that 
\begin{enumerate}
\item $\itGamm_{\mc{V}(\widehat{\bQ})} \unit \otimes \itGamm_{\mc{V}(\bP)} \unit \cong \itGamm_{\mc{V}(\bP)} \unit$ and 
\item $L_{\mc{Z}(\widehat{\bQ})} \unit \otimes L_{\mc{Z}(\bP)} \unit \cong L_{\mc{Z}(\bP)}\unit$,
\end{enumerate}
by the definitions of $\itGamm_{\widehat{\bQ}}$ and $\itGamm_{\bP}$. By Lemma \ref{lemma:commutes_with_Gamma_L}, to prove (1) and (2) it is enough to show
\begin{enumerate}
    \item[(3)] $\Theta(\mc{V}(\bP)) \subseteq \Theta(\mc{V}(\widehat{\bQ})$, and 
    \item[(4)] $\Theta(\mc{Z}(\bP))\supseteq \Theta(\mc{Z}(\widehat{\bQ}))$.
\end{enumerate}
For (3): suppose $B \in \bK^c$ and $\suppB(B) \subseteq \mc{V}(\bP).$ Since $\sigma(B)=\eta^{-1}(\suppB(B))$, this implies that $\sigma(B) \subseteq \eta^{-1}(\mc{V}(\bP)),$ which implies that $\sigma(B) \subseteq \mc{V}(\widehat{\bQ})$ from the assumption that $\eta^{-1}(\mc{V}(\bP)) = \mc{V}(\widehat{\bQ})$. By definition, this implies that $B \in \Theta(\mc{V}(\widehat{\bQ}))$. 

For (4), recall that $\bQ$ lies in the closure of $\{\bP\}$. If $ C\in \Theta(\mc{Z}(\widehat{\bQ}))$, then $\sigma(C)\subseteq \mc{Z}(\widehat{\bQ})$. We need to show that $\suppB(C)\not\ni \bP$. But if it were not so, then $\suppB(C)\supseteq \mc{V}(\bP)\ni\bQ$, and then $\sigma(C)=\eta^{-1}(\suppB(C))\ni \widehat{\bQ}$, a contradiction.
\end{proof}

\bex{tensorial_example_2}
Consider the following example. We may choose $\bK^c$ so that $\Spc \bK^c$ is homeomorphic to the space of four points $\{a,b,c,d\}$ where $d$ is closed, $b$ and $c$ specialize to $d$, and $a$ is a generic point. Then we can consider $X$ to be the subspace consisting of preimages of $b$, $c$, and $d$; let $\eta: X \to \Spc \bK^c$ be the inclusion map, so that $\eta$ induces the following map on the level of specialization posets:
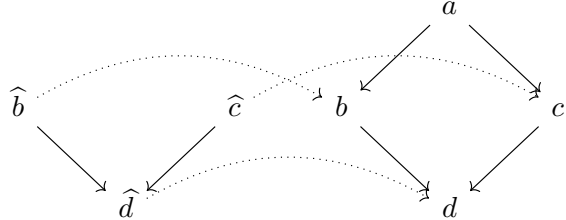
\[\begin{tikzcd}
                                                      &                                            &                                                       &              & a \arrow[ld] \arrow[rd] &              \\
\widehat{b} \arrow[rd] \arrow[rrr, dotted, bend left] &                                            & \widehat{c} \arrow[ld] \arrow[rrr, dotted, bend left] & b \arrow[rd] &                         & c \arrow[ld] \\
                                                      & \widehat{d} \arrow[rrr, dotted, bend left] &                                                       &              & d                       &             
\end{tikzcd}\]
We have a realizing faithful tensorial support $\sigma$ induced by $\eta$. By Theorem \ref{thm:O_2_faithful}, the extended support $\widetilde{\sigma}$ is not faithful, since the preimage $\eta^{-1}(\mc{V}(a))=X$ does not have a unique generic point.
\eex

\bex{tensorial-ex-3}
Consider \exref{tensorial_example_2} but where we now take $X$ to only have preimages of $b$ and $d$. Then $\eta$ induces the following map of specialization posets:
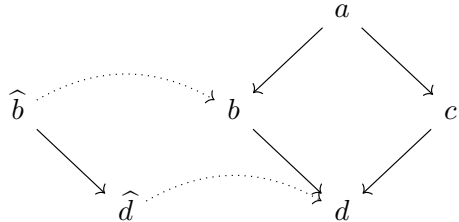
\[\begin{tikzcd}
                                                     &                                           &              & a \arrow[ld] \arrow[rd] &              \\
\widehat{b} \arrow[rd] \arrow[rr, dotted, bend left] &                                           & b \arrow[rd] &                         & c \arrow[ld] \\
                                                     & \widehat{d} \arrow[rr, dotted, bend left] &              & d                       &             
\end{tikzcd}\]
In this case, $\widetilde{\sigma}$ is faithful by Theorem \ref{thm:O_2_faithful}, since $\eta^{-1}(\mc{V}(a))=X$, $\eta^{-1}(\mc{V}(b))=\{\widehat{b}, \widehat{d}\},$ $\eta^{-1}(\mc{V}(c))=\{\widehat{d}\}$ and $\eta^{-1}(\mc{V}(d))=\{\widehat{d}\}$ each have a unique generic point.
\eex

\bre{etainverse-support}
In this situation, where $\sigma$ is faithful, tensorial, and realizing, there is a second possible extension of the support $\sigma$ other than $\widetilde{\sigma}$, namely, 
\[
\widehat{\sigma}(A):=\eta^{-1}(\SuppB(A)).
\]
It is straightforward to verify that $\widehat{\sigma}$ is a tensorial extension of $\sigma$. However, $\widehat{\sigma}$ is faithful if and only if $\eta$ is a homeomorphism; this can be seen directly or deduced from Theorem~\ref{nvy-restated}. In fact, there is a containment 
\[
\widehat{\sigma}(A) \subseteq \widetilde{\sigma}(A)
\] for all $A$. 
\ere

\section{Faithfulness of extended support induced by a surjective map}
\label{sect-surj}

In this section, we assume as usual that $\bK$ is a big monoidal triangulated category. We also assume throughout the section that $\bK^c$ is $\Spc$-Noetherian. We will prove that in this case support data on $\bK^c$ induced by surjective maps admit faithful extensions. 

We begin with a definition.
\bde{induced-supp}
Let $\rho: \Spc \bK^c \to X$ be a surjective continuous map of topological spaces. We say that the map
\[
\sigma(A):=\overline{\rho(\suppB(A))}
\]
is the {\it{support induced by $\rho$.}}
\ede

Note that any comparative support is induced by a surjective map, since it is induced by the comparison map $\rho$ by Lemma \ref{comparative-props}, parts (1) and (2); the main difference between the assumptions in this section and those previously in Section \ref{sect-extcomp} is that we have dropped the assumption that $\sigma$ is tensorial, but have added the assumption that $\bK^c$ is $\Spc$-Noetherian. The purpose is for eventual application to the central cohomological support in the subsequent section (Section \ref{sect-centralcoh}); for a finite tensor category $\bC$, both the questions of Noetherianity of $\Spc \ul{\bC}$ and the tensoriality of the central support are open.

\ble{ind-supp-is-supp}
Let $(X,\sigma)$ be the support induced by a surjective continuous map $\rho: \Spc \bK^c \to X$ on a small $\Spc$-Noetherian monoidal triangulated category $\bK^c$. Then $(X,\sigma)$ is a faithful and realizing support datum. 
\ele

\begin{proof}
    It is straightforward that $(X,\sigma)$ is a support datum. Faithfulness follows from the fact that $\suppB$ is faithful (\exref{balmer-supp}). Suppose $V \subseteq X$ is a closed subset. Then $\rho^{-1}(V)$ is a closed subset of $\Spc \bK^c$, and so there exists $A \in \bK^c$ with $\suppB(A) = \rho^{-1}(V)$. Surjectivity of $\rho$ then implies that $\sigma(A)=V$. 
\end{proof}

\bre{non-tensorial-surj}
It is clear that a support induced by a surjective map need not be tensorial. For instance, if $\Spc \bK^c$ is the topological space consisting of two distinct closed points, $X$ is the one-point space, and $\rho$ is the unique continuous map $\Spc \bK^c \to X$. Then the support induced by $\rho$ is not tensorial: if $x$ is the unique point of $X$, and if $A$ and $C$ are two objects of $\bK^c$ with disjoint nonempty supports (which exist, by Theorem \ref{nvy-restated}), then we have $\sigma(A)=\{x\}=\sigma(C)$, but also have 
\[
\bigcup_{B \in \bK^c} \suppB(A \otimes B \otimes C) = \suppB(A) \cap \suppB(C) = \varnothing,
\]
hence $A \otimes B \otimes C \cong 0$ for all objects $B$. But then
\[
\bigcup_{B\in \bK^c} \sigma(A \otimes B \otimes C) = \varnothing \subsetneq \sigma(A) \cap \sigma(B)=\{x\}.
\]
\ere

We now proceed in computing the extended Balmer supports of the Rickard idempotents associated to points of $X$.

\begin{proposition}
    \label{supp-gammav}
Let $(X,\sigma)$ be the support induced by a surjective map $\rho$ on a small $\Spc$-Noetherian monoidal triangulated category $\bK^c$, and let $V \subseteq X$ be a specialization-closed subset. We have
\[\SuppB(\itGamm_V \unit)=\rho^{-1}(V).\] 
\end{proposition}

\begin{proof}
    Indeed, by \leref{rickard-idem-supp}, 
\begin{align*}
    \SuppB(\itGamm_{V} \unit) &= \Phi_{\suppB} (\Theta(V))\\
    &= \bigcup_{\{A \in \bK^c :\sigma(A) \subseteq V\}} \suppB(A)\\
    &= \bigcup_{\{A \in \bK^c :\overline{\rho(\suppB(A))} \subseteq V\}} \suppB(A)\\
    &=\bigcup_{\{A \in \bK^c :\rho(\suppB(A)) \subseteq V\}} \suppB(A)\\
    &=\bigcup_{\{A \in \bK^c :\suppB(A) \subseteq \rho^{-1}(V)\}} \suppB(A)\\
    &= \rho^{-1}(V).
\end{align*}
The last line follows since Balmer support satisfies realization, using the Noetherianity assumption. 

\end{proof}

\begin{corollary}
\label{supp-gammax}
Let $(X,\sigma)$ be the support induced by a surjective map $\rho$ on a small $\Spc$-Noetherian monoidal triangulated category $\bK^c$, and let $x \in X$. We have
\[\SuppB(\itGamm_x \unit)=\rho^{-1}(x).\]
\end{corollary}

\begin{proof}
Let $\bP \in \Spc \bK^c$. Since $\itGamm_x \unit = \itGamm_{\mc{V}(x)} \unit \otimes L_{\mc{Z}(x)} \unit$, and since the extended Balmer support is tensorial by \prref{tensor-balmer}, it is clear that if $\bP$ is in $\SuppB(\itGamm_x\unit)$ then $\bP \in \SuppB(\itGamm_{\mathcal{V}(x)}\unit)$ and $\bP \not \in \SuppB(\itGamm_{\mathcal{Z}(x)} \unit)$ (recall \leref{rickard-idem-supp}). Since $\mc{V}(x) \cap (X \backslash \mc{Z}(x)) =\{x\},$ the proposition follows from Proposition \ref{supp-gammav}. 
\end{proof}

Since $\sigma$ is not tensorial, we cannot assume that $\widetilde{\sigma}$ is actually an extension of $\sigma$, at this point, since \prref{ext-supp-is-ext-supp} requires tensoriality. However, we are able to prove the following description of $\widetilde{\sigma}$, which will then allow us to describe precisely when $\widetilde{\sigma}$ agrees with $\sigma$ on compact objects.

\begin{theorem}
    \label{surj-map-faith}
    Let $(X,\sigma)$ be the support datum induced by a surjective map $\rho$ on a $\Spc$-Noetherian small monoidal triangulated category $\bK^c$. Then 
    \[
    \rho(\SuppB(A)) = \widetilde{\sigma}(A) 
    \]
    for all $A \in \bK$. As a consequence, $\widetilde{\sigma}$ is faithful. It is an extended support if and only if the map $\rho$ is closed.
\end{theorem}

\begin{proof}
    For the containment $\subseteq$, let $x = \rho(\bP)$ for some $\bP \in \SuppB(A).$ Then $\bP \in \SuppB(\itGamm_x \unit)$ by \Cref{supp-gammax}. Since $\SuppB$ is tensorial, this implies $\bP \in \SuppB(\itGamm_x\unit \otimes A)$. This implies that $\itGamm_x\unit \otimes A \not \cong 0$, and hence $x \in \widetilde{\sigma}(A)$. 

    For the containment $\supseteq$, if $x \in \widetilde{\sigma}(A)$, then we have $\itGamm_x \unit\otimes A \not \cong 0$. Since $\SuppB$ is faithful by \coref{balmer-faith}, $\SuppB(\itGamm_x \unit\otimes A) = \SuppB(\itGamm_x\unit) \cap \SuppB(A) = \rho^{-1}(x) \cap \SuppB(A)$ is nonempty. In particular, there exists $\bP$ in this intersection. But then $x=\rho(\bP)$, and this $\bP$ is in $\SuppB(A)$, so $x \in \rho(\SuppB(A))$.

    If $A$ is compact, we have now shown that 
    \begin{align*}
        \widetilde{\sigma}(A) &= \rho(\SuppB(A))\\
        &= \rho(\suppB(A)).
    \end{align*}Clearly, faithfulness of $\widetilde{\sigma}$ follows then from the faithfulness of $\SuppB$. 

    If $\rho$ is closed, we have
    \[
    \widetilde{\sigma}(A)= \rho(\SuppB(A)) = \rho(\suppB(A)) = \sigma(A)
    \]
    for all $A$ compact, so $\widetilde{\sigma}$ extends $\sigma$. On the other hand, if $\rho$ is not closed, then there exists some closed set $\suppB(A)$ with 
    \[
    \widetilde{\sigma}(A) = \rho(\SuppB(A)) = \rho(\suppB(A)) \subsetneq \overline{\rho(\suppB(A))} = \sigma(A)
    \]
    for some compact $A$, using the fact that $\SuppB$ extends $\suppB$ and that $\suppB$ satisfies realization, since $\Spc \bK^c$ is Noetherian (\exref{balmer-supp}).
\end{proof}

\section{Applications to stable categories of finite tensor categories}
\label{sect-centralcoh}

In this section, let $\bC$ be a finite tensor category as in \cite{EGNO2015}. That is, $\bC$ is an abelian $\kk$-linear rigid monoidal category for some algebraically-closed field $\kk$, each space of morphisms between objects of $\bC$ is finite-dimensional, $\bC$ has finitely many simple objects, all objects of $\bC$ have finite length, and the unit object $\unit$ is simple. In this case, the stable category $\ul{\bC}$ of $\bC$ is defined as having the same objects as $\bC$, and morphism spaces
\[
\Hom_{\ul{\bC}}(A,B) := \Hom_{\bC}(A,B) / \PHom_{\bC}(A,B),
\]
where $\PHom_{\bC}(A,B)$ denotes the collection of morphisms $A \to B$ in $\bC$ which factor through a projective object. The stable category $\ul{\bC}$ is monoidal triangulated \cite{Happel1988}. In fact, $\ul{\bC}$ is a small monoidal triangulated category: it appears as the compact part of the rigidly-compactly-generated monoidal triangulated category $\ul{\Ind(\bC)}$ \cite[Appendix A]{NVY3}. We will denote
\[
\bK:=\ul{\Ind(\bC)}, \;\; \bK^c :=\ul{\bC}.
\]
Since $\bC$ has finitely many simple objects, $\ul{\bC}$ is finitely-generated: the generator is the image of the direct sum of the simples.

\bex{hopf}
One of the motivating examples of finite tensor categories is the category of finite-dimensional representations of a finite-dimensional Hopf algebra $H$. In this case, the category $\Ind(\bC)$ is the category of all representations of $H$.
\eex

Recall that the {\it{cohomology ring}} $R^\bullet$ of $\bC$ is the $\kk$-algebra
\[
R^\bullet:=\bigoplus_{i \geq 0} \Ext^i_{\bC}(\unit,\unit)
\]
under the Yoneda product. The cohomology can also be described via the triangulated structure of $\bK^c$ (see e.g. \cite[Proposition 2.6.2]{CTVZ2003}): 
\[
R^\bullet \cong \bigoplus_{i \geq 0} \Hom_{\bK^c} (\unit, \Sigma^i \unit).
\]
The {\it{categorical center}} $C^\bullet \subseteq R^\bullet$ was introduced in \cite{NVY3}, and is defined as the subalgebra of $R^\bullet$ generated by homogeneous elements $g: \unit \to \Sigma^i \unit$ such that the diagrams
\[
\begin{tikzcd}
\unit \otimes A \arrow[d, "g \otimes \id_A"] \arrow[r, "\cong"] & A          & A \otimes \unit \arrow[l, "\cong"] \arrow[d, "\id_A \otimes g"] \\
\Sigma^i \unit \otimes A \arrow[r, "\cong"]                     & \Sigma^i A & A \otimes \Sigma^i \unit \arrow[l, "\cong"]  
\end{tikzcd}
\]
commute for all simple objects $A$. 

By composing shifts of morphisms as in \cite[Section 1.3]{NVY3}, there is a canonical action of the categorical center $C^\bullet$ on each algebra $\Hom^\bullet(A,A) := \bigoplus_{i \geq 0} \Hom_{\bK^c} (A, \Sigma^i A).$ We say that $\bK^c$ satisfies the {\it{weak finite generation condition}}, or (wfg) condition, if each $\Hom^\bullet(A,A)$ is finitely-generated as a $C^\bullet$ module. It would be a consequence of the Etingof--Ostrik finite-generation conjecture (see \cite{EO2004}, originally posed as a question by Friedlander and Suslin in \cite{FriedlanderSuslin1997}) that $\bK^c=\ul{\bC}$ would satisfy the (wfg) condition, for each finite tensor category $\bC$ (see \cite[Section 1.5]{NVY3}). 

We have the {\it central cohomological support} defined by
\[
\suppC(A):=\{ \mf{p} \in \Proj C^\bullet : I(A) \subseteq \mf{p}\}
\]
where $I(A)$ is the annihilator ideal of $\Hom^\bullet(A,A)$ in $C^\bullet$. The central cohomological support is a support datum (see \cite[Proposition 6.1.1]{NVY3}). As long as $\bK^c$ satisfies the (wfg) condition, the central cohomological support is faithful, see the discussion in \cite[Corollary 7.1.3]{NVY3}, which uses the result \cite[Corollary 4.2]{BPW2021} that connects dimension of a support variety with the complexity of the module. 

When $\bK^c$ satisfies the (wfg) condition, there exists a surjective continuous map $\Spc \bK^c \to \Proj C^\bullet$
given by
\[\rho(\bP):=\langle g \in C^\bullet : g \text{ homogeneous,} \cone(g) \not \in \bP\rangle
\]
for any prime ideal $\bP \in \Spc \bK^c$, see \cite[Corollary 7.1.3, Theorem 7.2.1]{NVY3}, a noncommutative analogue of Balmer's comparison map \cite{Balmer2010}. If the central support $\suppC$ is tensorial,
then we have the universal map
\[
\Proj C^\bullet \xrightarrow{\eta} \Spc \bK^c.
\]
In this case $\rho \circ \eta= \id_{\Proj C^\bullet}$ and $(\eta \circ \rho)(\bP) \subseteq \bP$ for all $\bP \in \Spc \bK^c$ by \cite[Proposition 8.1.1]{NVY3}. Whether or not $\suppC$ is tensorial in general is still an open question.

\bnota{centralcohom-nota}
    The extension of the central cohomological support $\suppC$ will be denoted $\SuppC$.
\end{notation}

Before stating our main theorem giving conditions under which the central cohomological support satisfies the necessary assumptions given in the previous sections for there to exist a faithful extension, we prove some preliminary results.

\bpr{var-tens-power}
Let $\bK^c=\ul{\bC}$ be the stable category of a finite tensor category $\bC$ satisfying weak finite generation, and suppose that the categorical center $C^\bullet$ is finitely-generated. Let $A \in \bK^c$ and let $G$ be the image of the direct sum of the simple objects of $\bC$ in $\bK^c.$ Then $\suppC(A \otimes G \otimes A) = \suppC(A).$
\epr

\begin{proof}
We have the containment 
\[
\suppC(A \otimes G \otimes A) \subseteq \suppC(A)
\]
by \cite[Proposition 6.1.1(e)]{NVY3}. Suppose that this containment is strict. Then there exists some homogeneous maximal ideal $\mf{m} \in \suppC(A)\subseteq \Proj C^\bullet$ which is not in $\suppC(A \otimes G \otimes A)$ since in schemes of locally finite type over a field $\kk$, closed subsets are determined by their closed points (see e.g. \cite[Proposition 3.35]{GW2010}). Take a collection of generators $g_1,..., g_n$ of $\mf{m}$. Set $B= \cone(g_1) \otimes... \otimes \cone(g_n)$. We have $\suppC(B)=\{ \mf{m}\}=\suppC(A \otimes B)$ by \cite[Proposition 7.1.1(a)]{NVY3}. 

Now note that for any objects $D$ and $E$ in $\bK^c$, we have a distinguished triangle
\[
D \otimes \unit \otimes E\xrightarrow{\id_D \otimes g_i \otimes \id_E}  D \otimes \unit \otimes E \to D \otimes \cone(g_i) \otimes E \to \Sigma(D \otimes \unit \otimes E).
\]
It follows that $\suppC(D \otimes \cone(g_i) \otimes E) \subseteq \suppC(D \otimes E)$. Since $B$ is a tensor product of all of the $\cone(g_i)$'s, we also have $\suppC(D \otimes B \otimes E)\subseteq \suppC(D \otimes E),$ for any objects $D$ and $E$. 

By the above discussion and  \cite[Proposition 7.1.1(d)]{NVY3},
we have
\begin{align*}
    \suppC( (A \otimes B) \otimes G \otimes (A \otimes B)) & \subseteq \suppC(A \otimes G \otimes A \otimes B)\\ 
    &= \suppC(A \otimes G \otimes A) \cap \suppC(B)\\
    &= \varnothing.
\end{align*}
It follows that $(A \otimes B) \otimes G \otimes (A \otimes B) \cong 0$ by the faithfulness of the central cohomological support, and hence $A \otimes B \cong 0$ by \cite[Proposition 4.1.1]{NVY2}. But this is a contradiction, since $A \otimes B$ has a nontrivial support variety. 
\end{proof}

Before proving that the comparison map $\rho$ is closed (allowing us to apply the general results of Section \ref{sect-surj}), we need one more technical result.

\bpr{coneg-product-middle}
Let $\bK^c=\ul{\bC}$ be the stable category of a finite tensor category $\bC$ satisfying weak finite generation. Let $G$ be the image of the direct sum of the simple objects of $\bC$ in $\bK^c$. Let $A\in \bK^c$ and $\mf{p} \in \suppC(A)$. Then $\mf{p}$ is in
\[
\suppC(A_1 \otimes G \otimes A_2 \otimes G \otimes\cdots \otimes A_n),
\]
where each $A_i$ is either $A$ or $\cone(g)$ for some $g \in \mf{p}$.
\epr

\begin{proof}
    We induct on $n$; clearly, if $n=1$, then the result holds, recalling again that $\suppC(\cone(g))=Z(g)$ for any $g \in C^{\bullet}$. Set $B=A_1 \otimes G \otimes A_2 \otimes G \otimes \cdots \otimes A_n$ as in the statement of the proposition. If all the objects $A_i$ are equal to $A$, then the result follows from \prref{var-tens-power}. Furthermore, if either $A_1$ or $A_n$ is equal to $\cone(g)$, for some $g \in \mf{p}$, then the claim follows from the inductive hypothesis and \cite[Proposition 7.1.1(d)]{NVY3}. Hence we can assume that 
    \[
    B=A' \otimes G \otimes C \otimes G \otimes A'',
    \]
    where $A'$ and $A''$ are both objects of the form $A \otimes G \otimes \cdots \otimes G \otimes A$, and $C$ has the form $\cone(g_1) \otimes G \otimes A_i \otimes G \cdots \otimes A_j \otimes \cone(g_2).$ Suppose without loss of generality that $A'$ has more tensorands than $A''$. By tensoring $B$ on the right, we can form the object
    \[
    D=(A' \otimes G \otimes C) \otimes G \otimes (A' \otimes G \otimes C).  
    \]
    But now note that by \cite[Proposition 4.1.1]{NVY2}, $D$ generates the same thick ideal as $A' \otimes G \otimes C$. Since $B$ is in the thick ideal generated $A' \otimes G \otimes C$, and since $D$ is in the thick ideal generated by $B$, we have
    \[
    \langle D \rangle = \langle B \rangle = \langle A' \otimes G \otimes C \rangle. 
    \]
    By the properties of support, $\suppC(A' \otimes G \otimes C) \subseteq \suppC(B)$. But $\mf{p} \in \suppC(A' \otimes G \otimes C),$ by the inductive hypothesis.
\end{proof}

\bpr{central-supp-rho}
Let $\bK^c=\ul{\bC}$ be the stable category of a finite tensor category $\bC$ satisfying weak finite generation, and suppose that the categorical center $C^\bullet$ is finitely-generated. We have $\rho(\suppB(A)) = \suppC(A)$ for all $A \in \bK^c$. 
\epr

\begin{proof}
    We know by \cite[Corollary 6.2.5(b)]{NVY3} that $\rho(\suppB(A)) \subseteq \suppC(A)$. For the other direction, suppose $\mf{p} \in \suppC(A)$. We will construct a prime ideal $\bP$ such that $\rho(\bP) = \mf{p}$ and $\bP \in \suppB(A)$. Consider the thick ideal
    \[
    \bI := \langle \cone(g) : g \not \in \mf{p} \rangle,
    \]
    and the collection of objects
    \[
    \bM := \{ A_1\otimes G \otimes A_2 \cdots \otimes G \otimes A_n : A_i \in \bS \},
    \]
    where $G$ is the sum of simple objects and $\bS$ is the set of objects which are either $A$ or $\cone(g)$ for some $g \in \mf{p}$. So long as $\bI$ and $\bM$ are disjoint, we may produce a prime ideal containing $\bI$ which intersects $\bM$ trivially, by \cite[Lemma A.1.1]{NVY4}, since $\bM$ has the property that given any two objects $B$ and $D$, there exists an object $C$ such that $B \otimes C \otimes D \in \bM$ (of course, we have constructed it so that $C$ can be chosen as $G$). Indeed, $\bI$ and $\bM$ must be disjoint, since by \prref{coneg-product-middle} we have $\suppC(\cone(g)) = Z(g)$ for any $g \in C^\bullet$, hence 
    \[
    \suppC(B) \subseteq \bigcup_{g \not \in \mf{p}}Z (g)
    \]
    for any $B \in \bI$. Therefore, $\mf{p} \not \in \suppC(B)$ for any $B \in \bI$. Since $\mf{p} \in \suppC(C)$ for all $C \in \bM$ by \prref{coneg-product-middle}, it follows that $\bI \cap \bM$ is empty. Thus, by \cite[Lemma A.1.1]{NVY4}, we can produce a prime ideal $\bP$ containing $\bI$ and disjoint from $\bM$. Since $\bP$ contains $\bI$ and is disjoint from $\bM$, we have
    \begin{align*}
        \rho(\bP) &= \langle g \in C^\bullet : g \text{ homogeneous}, \cone(g) \not \in \bP \rangle\\
        & = \langle g \in C^\bullet : g \text{ homogeneous }, g \in \mf{p}\rangle\\
        &=\mf{p}.
    \end{align*}
    Additionally, since $\bP$ is disjoint from $\bM$, it in particular does not contain $A$, that is, $\bP \in \suppB(A)$, and we are done.

\end{proof}

Immediately from \prref{central-supp-rho}, we deduce:

\begin{corollary}
    \label{rho-closed-fintenscat}
Let $\bK^c=\ul{\bC}$ be the stable category of a finite tensor category $\bC$ satisfying weak finite generation, and suppose that the categorical center $C^\bullet$ is finitely-generated and that $\Spc \bK^c$ is Noetherian. Then the map $\rho$ is a closed map.
\end{corollary}

We have now completed the work necessary to state our main results for central support on the stable category of a finite tensor category.

\begin{theorem}
    \label{stablecat-summary}
    Let $\bK^c=\ul{\bC}$ be the stable category of a finite tensor category $\bC$ satisfying the weak finite-generation condition. Then if $\bK^c$ satisfies either of the following two conditions, the central support $\suppC$ admits a faithful extension:
    \begin{enumerate}
        \item $\Proj C^\bullet$ is Noetherian and $\suppC$ is tensorial;
        \item $C^\bullet$ is finitely-generated and $\Spc \bK^c$ is Noetherian.
    \end{enumerate}
\end{theorem}

\begin{proof}
(1): If $\suppC$ is tensorial and $\Proj C^\bullet$ Noetherian, then the map $\rho$ defined above satisfies the axioms of a comparison map given in \deref{comparative}, so that $\suppC$ is a comparative support and we can deduce that it has a faithful extension, by Theorem \ref{comp-faith}.

(2): We know that central support $\suppC$ is induced by the surjective map $\rho$ by \prref{central-supp-rho}. Hence if $\Spc \bK^c$ is Noetherian, the result follows from \Cref{surj-map-faith}.
\end{proof}

\bre{bik-comparison}
Benson--Iyengar--Krause construct extensions in \cite{BIK2008} for supports arising from an arbitrary graded-commutative Noetherian ring $R$ acting on a triangulated category. However, this support is based on the homogeneous prime spectrum of $R$ rather than $\Proj R$, in other words, the irrelevant ideal (if the ring is $\mathbb{N}$-graded) is allowed as a point. Other than this difference, the central cohomological support fits into their framework, and our extension is defined in the same way as theirs (indeed, their extension serves as one of the primary models for the extension in this paper). In \cite[Theorem 5.2]{BIK2008}, they prove that their extension is faithful. This does not immediately imply that the extension for the $\Proj$-based version is faithful, however, since a priori their theorem still allows for a nonzero object in $\bK$ to have support equal to the irrelevant ideal, which would translate to empty $\SuppC$-support. We may therefore view \Cref{stablecat-summary} as a strengthening of the faithfulness result of Benson--Iyengar--Krause under tensoriality or $\Spc$-Noetherianity assumptions.
\ere

\bex{quantum-complete-intersection}

Let $\ell\ge 2$ be an integer and $\kk$ be a field of characteristic either 0 or relatively prime to $\ell$. Let $q$ be a primitive $\ell$th root of unity. There is an action of the elementary abelian group $G:=(\mathbb{Z}/\ell \mathbb{Z})^{\times n}$ on the truncated polynomial ring $R:=\kk[x_1,..., x_n]/(x_1^\ell,..., x_n^\ell)$ given by letting the $i$th generator of $G$ send $x_j$ to $q^{\delta_{ij}} x_j.$ The resulting smash product $A:=R \rtimes G$ is a noncommutative, noncocommutative Hopf algebra, and is referred to as a quantum elementary abelian group; for background details, see \cite[\S2]{pevtsova2009varieties}. It can alternatively described as the Borel subalgebra of the small quantum group $u_q(\mathfrak{sl}_2^{\oplus n})$. Set $\bK \coloneqq \ul{\Mod}(A)$. Pevtsova--Witherspoon pose the question in \cite[Introduction]{pevtsova2015tensor}: do support varieties for infinite-dimensional $A$-modules satisfy a tensor product property? We proceed to use this example to illustrate the techniques developed above, providing an affirmative answer to the question of Pevtsova--Witherspoon, which gives an alternate proof of a result of Negron--Pevtsova \cite{NP2023a}.

Pevtsova and Witherspoon prove in \cite[Theorem~1.1]{pevtsova2015tensor} that the cohomological support of $\bK^c$ is tensorial, and hence we have a continuous map
\[
\eta: \Proj \cohom(A,\kk) \to \Spc \bK^c
\]
by \cite[Theorem 1.3.1]{NVY2022}. By \cite[Lemma 3.1]{pevtsova2015tensor}, simple $A$-modules tensor-commute with all objects of $\bK^c$, hence $C^\bullet = \cohom(A,\kk)$ by the definition of $C^\bullet$. The cohomology ring is finitely-generated: $\cohom(A,\kk)\cong \kk[y_1,\dots, y_n]$ is a polynomial ring (where each $\deg (y_i)=2$), see for example the straightforward computation in \cite[\S4]{pevtsova2009varieties}. The map 
\[
\rho: \Spc \bK^c \to \Proj \cohom(A,\kk)
\]
recalled above is a left inverse to $\eta$, and is closed by \prref{central-supp-rho}. By the classification of thick ideals for $\bK^c$, also due to Pevtsova--Witherspoon \cite[Theorem 1.2]{pevtsova2015tensor}, $\rho$ must be injective, and hence $\eta$ and $\rho$ are inverse homeomorphisms. Indeed, using the explict formula for $\rho$, if $\rho(\bP) = \rho(\bQ)$ for $\bP, \bQ \in \Spc \bK^c$, then it would follow that for all $g \in \cohom(A,\kk)$, we would have $\cone(g) \in \bP \Leftrightarrow \cone(g) \in \bQ$, which would imply that the supports of $\bP$ and $\bQ$ would be the same, by \cite[Proposition 7.1.1]{NVY3}, a contradiction. Therefore 
\[
\Spc \bK^c \cong \Proj \cohom(A,\kk) \cong \mathbb{P}^{n-1}.
\]

These results allow us to apply Theorem~\ref{stablecat-summary}, since $\Proj C^\bullet$ is Noetherian, $\suppC$ is tensorial, $C^\bullet$ is finitely-generated, and $\Spc \bK^c$ is Noetherian. We obtain that $\suppC$ admits a faithful extension to all of $\bK$. Furthermore, since $\suppC$ identifies with the Balmer support, by \prref{tensor-balmer} we have tensoriality of $\SuppC$: the question posed by Pevtsova--Witherspoon, asking whether there exists an extension of $\suppC$ satisfying a tensor product property, has a positive answer.

We note that this problem was first solved in \cite{NP2023a} using an alternate method. Negron--Pevtsova construct a different extended support, termed {\em{hypersurface support}}, and prove in \cite[Theorem 7.9]{NP2023a} that this hypersurface support is a faithful extended support satisfying a version of the tensor product property. However, for hypersurface support to be defined, one needs to have a smoothly integrable Hopf algebra. We emphasize that the techniques we use to recover this result above make no reference to the smooth integration; nevertheless, even though the Negron--Pevtsova is constructed in a different way from the support we provide here, by Proposition \ref{uniqueness_of_balmer_extension} we know that the two extended supports are in fact the same. By \prref{tensor-balmer} and \coref{balmer-faith} we expect existence of a faithful and tensorial extended support to be a general phenomena in monoidal triangular geometry, and not specific to smoothly integrable Hopf algebras.
\eex

\bibliography{bibl}
\bibliographystyle{myamsalpha}


\end{document}